\documentclass[english,leqno,10pt,a4paper]{article}
\usepackage[totalheight=23 true cm, totalwidth=16 true cm]{geometry}

\usepackage{amsmath,amstext, amsthm, amssymb}

\usepackage{pifont}
\usepackage{hyperref}
\usepackage{srcltx}
\usepackage[english]{babel}
\usepackage[all]{xy}
\usepackage{color}

\newcommand{\Appendix}{\appendix
\section*{Appendix}
\addcontentsline{toc}{section}{Appendix}}

\usepackage[color=orange,textwidth=1.5cm]{todonotes}
\usepackage{marginnote}
\definecolor{luciacolor}{rgb}{0.01,0.28,1.00}
\def\lucia#1{#1}
\def\charl#1{#1}

\newtheorem{thm}{Theorem}[section]
\newtheorem{cor}[thm]{Corollary}
\newtheorem{lemma}[thm]{Lemma}
\newtheorem{prop}[thm]{Proposition}

\theoremstyle{remark}
\newtheorem{rmk}[thm]{Remark}

\theoremstyle{definition}
\newtheorem{defn}[thm]{Definition}

\newtheorem{exa}[thm]{Example}
\newtheorem{notation}[thm]{Notation}

\numberwithin{equation}{section}

\def\beq{\begin{equation}}
\def\eeq{\end{equation}}

\def\N{{\mathbb N}}
\def\Z{{\mathbb Z}}
\def\Q{{\mathbb Q}}

\def\C{{\mathbb C}}

\def\I{{\mathbb I}}

\def\l{\left}
\def\r{\right}
\def\[[{\l[\l[}
\def\]]{\r]\r]}
\def\p{\prime}

\def\sgq{{\sigma_q}}
\def\sgqp{({\sigma_q},\partial)}

\def\Sgq{\Sigma_q}

\def\cf{\emph{cf. }}
\def\ie{i.e.}

\def\cC{{\mathcal C}}

\def\cF{{\mathcal F}}
\def\cM{{\mathcal M}}
\def\cN{{\mathcal N}}

\def\cR{{\mathcal R}}
\def\cS{{\mathcal S}}

\def\cV{{\mathcal V}}
\def\cW{{\mathcal W}}
\def\cX{{\mathcal X}}
\def\cY{{\mathcal Y}}

\def\wtilde{\widetilde}

\def\ul{\underline}

\def\om{\omega}

\def\GL{\mathop{\rm GL}}
\def \Alg{\mathop{\rm Alg}}

\def\Mer#1{{\cM(#1)}}
\def\Diff{{\rm DiffMod}}
\def\Hom{{\rm Hom}}
\def\Proj{{\rm Proj}}
\def\Aut{{\rm Aut}}
\def\Isom{{\rm Isom}}
\def\Vect{{\rm Vect}}

\def\Gal{{\rm Gal}}

\def\$$endproof{\eqno{\qedhere}$$\end{proof}}

\newenvironment{itemize*}%
  {\begin{itemize}%
    \setlength{\itemsep}{0pt}%
    \setlength{\parskip}{0pt}}%
  {\end{itemize}}

  \newenvironment{enumerate*}%
  {\begin{enumerate}%
    \setlength{\itemsep}{0pt}%
    \setlength{\parskip}{0pt}}%
  {\end{enumerate}}

\title{Galois theories of $q$-difference equations: comparison theorems}

\author{Lucia Di Vizio\footnote{Lucia DI VIZIO,
Laboratoire de Math\'ematiques UMR 8100, CNRS,
Universit\'e de Versailles-St Quentin,
45 avenue des \'Etats-Unis
78035 Versailles cedex, France. {\tt lucia.di.vizio@math.cnrs.fr}}
~and Charlotte Hardouin\footnote{Charlotte HARDOUIN, Institut de Math\'{e}matiques de Toulouse,
118 route de Narbonne,
31062 Toulouse Cedex 9, France. {\tt hardouin@math.univ-toulouse.fr}}}

\date{\today}

\begin{document}

\bibliographystyle{amsalpha}


\maketitle

\begin{abstract}
We establish some comparison results among the different parameterized Galois theories for $q$-difference equations, completing the work
\cite{ChatHardouinSinger}, that addresses the problem in the case without parameters.
Our main result is the link between the abstract parameterized Galois theories, that give information on the differential properties of abstract solutions of $q$-difference equations, and the properties of meromorphic solutions of such equations. Notice that a linear $q$-difference equation with meromorphic coefficients always admits a basis of meromorphic solutions, as proven in \cite{Praag2}.
\end{abstract}


\section*{Introduction}

A linear $q$-difference system is a
linear functional equation of the form $\vec{y}(qx)=A(x)\vec{y}(x)$, where, to fix ideas, we assume that
$A(x) \in \GL_\nu(\C(x))$ and that $q \in \C^*:=\C\smallsetminus\{0\}$ is not a root of unity.
This kind of functional equations appears in the literature for many reasons, for instance: they are discretization of
linear differential equations that can be recovered letting  $q \to 1$  in $\frac{\vec{y}(qx)-\vec{y}(x)}{(q-1)x}$;
they appear in combinatorial problems in relation with $q$-series; they have a geometric interpretation as functional equations on the torus
$\C^\ast/q^\Z$, when $|q|\neq 1$.
\par
A {functional equation} is useful as far as it allows to grasp properties of its solutions. To achieve this purpose, there are two different schools:
some may try to actually find explicit solutions to the equation, while others concentrate on its structural properties, that could give information
on solutions that we are unable or unwilling (but most of the time unable) to write.
Galois theory of difference equations follows the second {line of thoughts}.
\par
Difference Galois theory has been introduced in \cite{Franke-DifferenceGaloisTheory}.
Since the first systematic work \cite{vdPutSingerDifference}, which follows quite an abstract point of view,
the Galois theory of $q$-difference equation has been developing in many directions.
Theories and theorems can be classified according to different criteria, for instance:
\begin{enumerate}
  \item After \cite[Theorem 3]{Praag2} any $q$-difference equations with meromorphic coefficients
  has a basis of meromorphic solutions. Therefore one can decide to use it to define a Galois theory,
  as in \cite{SauloyENS}, or one can define abstract solutions as in \cite{vdPutSingerDifference}.
  \item
  One can define the Galois group of $\vec{y}(qx)=A(x)\vec{y}(x)$ as a group of automorphisms of a convenient extension containing a full set of solutions or use a more abstract Tannakian point of view.
\end{enumerate}
Following these different points of view, a certain number of  Galois groups have been defined in the literature.
In \cite{ChatHardouinSinger}, the problem of establishing the isomorphisms among them is addressed. See Theorem \ref{thm:compalgpv} below.

\par
In \cite{HardouinSinger}, the authors develop a parameterized Galois theory that
takes into account the action of a derivation on a set of solutions of a $q$-difference equations.
One of the drawbacks of this theory is that it is based on an abstract set of solutions and their derivatives, constructed
in an algebra over  a differentially closed extension of $\C$, therefore quite a huge field of definition.
A natural question is to compare these abstract solutions with the Galois theory constructed using a basis of meromorphic solutions of the system.
\par
To complete the picture, one can attach to any $q$-difference system several groups thanks to the theory of differential Tannakian categories introduced by
Kamensky  \cite{kamtan}
and Ovchinnikov  in  \cite{DifftanOv}.
A differential Tannakian category is a  generalization of the notion of Tannakian category introduced by Deligne \cite{Delgrotfest}.
As a  Tannakian category (with a fiber functor) is isomorphic   to  the category of representations of a linear algebraic groups, a  differential Tannakian category
(with a differential fiber functor) is isomorphic to  the category of representations of a linear differential algebraic group (see Appendix \ref{appendix:diffgeometry}).
Any full set of solutions determines a neutral differential fiber functor on the category generated by the associated $q$-difference module,
and hence a group.
Moreover, the differential Tannakian setting allows to introduce the so called intrinsic Galois groups, associated to non-neutral fiber functors.
\par
The main point of this paper is to establish the isomorphisms between all these groups in the literature, which should allow for a more fluid
application of $q$-difference Galois theory, being able to exploit the different advantages of each point of view.
In \cite{peonnietodiffcomp}, the author compared some
parameterized Galois groups defined over some  abstract base fields and obtained
the analogue of Theorem \ref{thm:weakextension} in her setting.
{In this paper, we are interested in
clarifying the relation between the meromorphic solutions and the abstract Galois theory, so that we can not rely on \cite{peonnietodiffcomp}.}

\smallskip
The paper is organized as follows. First we define the different kind of solutions and the algebras generated by them. This allows for a
definition of the Galois group as a group of automorphisms.
Then we recall the (differential) Tannakian formalism and \lucia{the different groups that such theory allow to associate to a $q$-difference system.}
We conclude comparing all the different groups defined in the paper.

\medskip
{\bfseries Acknowledgments.} We are indebted to several colleagues whose interest for this paper has not faded during its long preparation.
We would like to thanks the referee for her or his attentive reading and the useful remarks.

{\small\tableofcontents}


\section{$q$-difference systems and their solutions}\label{sec:qdiffsystemsolutions}

Let $K$ be a field of characteristic zero and $q\neq 0,1$ be a fixed element of $K$.
The field $K(x)$ is naturally a {$q$-difference field}, \ie, it is equipped with the
\index{$q$-difference}
\index{$q$-difference!field}
{$q$-difference operator}
\index{$q$-difference!operator}
$$
\begin{array}{rccc}
\sgq:&K(x)&\longrightarrow&K(x),\\
&f(x)&\longmapsto&f(qx).
\end{array}
$$
More generally, we will call $q$-difference field one of the following pairs:
\begin{enumerate*}
\item
a field extension
$\cF/K(x)$, with a field automorphism extending
the action of $\sigma_q$, which we will also call $q$-difference operator and denote by $\sgq$.
\item
any sub-field stable by $\sgq$ and $\sgq^{-1}$ of the fields considered at the previous item, with the restriction of $\sgq$.
\end{enumerate*}
We denote by $\cC$ (or sometimes $\cF^{\sgq}$) the field of $\sgq$-constants of $\cF$,
\ie, the subfield of the elements of $\cF$ fixed by $\sgq$.
The previous definition implies that $\cC$ with the identity automorphism is
a $q$-difference field, that we will call trivial.

\begin{exa}\label{exa:qdifffield}
Typical examples of $q$-difference extensions of $K(x)$ are:
\begin{trivlist}
\item 1.
the field of formal Laurent series $K((x))$, equipped with the automorphism
$\sgq(\sum_n a_nx^n)=\sum_n a_n q^ nx^n$;

\item 2.
the field  $K(x^{1/r})$, for $r\in\Z_{>1}$, equipped with $\sgq(x^{1/r})=\tilde{q}x^{1/r}$,
for a chosen $r$-th root $\tilde{q}$ of $q$;

\item 3.
for $K=\C$, the  field of meromorphic functions over
$\C^*=\C\smallsetminus\{0\}$.
\end{trivlist}
\end{exa}

A linear $q$-difference system (of order $\nu$) is a functional equation of the form
\beq\label{eq:sys}
\sgq(\vec{y})=A\vec{y},
\eeq
where $A$ is an invertible square matrix of order $\nu\geq 1$ with coefficients in a $q$-difference field $(\cF,\sgq)$ as above,
i.e., $A\in\GL_\nu(\cF)$, and $\vec{y}$ is a vector of unknowns. The solutions vectors are to be found in a $q$-difference extension of $(\cF,\sgq)$.
It is well known that they generate a vector space over $\cC$ of order at most $\nu$.

\smallskip
In the text below, we  will consider
the action of a derivation $\partial$ on the solution set of \eqref{eq:sys}.
To do so, we will assume that there exists a
derivation $\partial$ of the field $\cF$, commuting to the action of $\sgq$,
that is $\sgq\circ\partial=\partial\circ\sgq$.
We will call $(\cF,\sgq,\partial)$ a $q$-difference differential field or a $\sgqp$-field, for short. Any $q$-difference fields of Example
\ref{exa:qdifffield} is a $\sgqp$-field endowed with the derivation $\partial=x\frac{d}{dx}$.
Notice that, any $q$-difference field $(\cF,\sgq)$ can be turned in a
$q$-difference differential field with the trivial derivation.

\begin{rmk}
\label{rmk:prolongation}
If $\vec{y}$ is a solution vector of \eqref{eq:sys}, then the commutativity of $\sgq$ and $\partial$ and the Leibnitz rule imply that:
$$
\sgq\begin{pmatrix}\partial(\vec{y})\\\vec{y}\end{pmatrix}
=\begin{pmatrix} A & \partial(A)\\ 0 & A\end{pmatrix}
\begin{pmatrix}\partial(\vec{y})\\\vec{y}\end{pmatrix},
$$
hence $\begin{pmatrix}\partial(\vec{y})\\\vec{y}\end{pmatrix}$ is the solution vector
of a $q$-difference system of order $2\nu$.
This quite trivial remark is at the origin of more abstract
constructions considered below.
\end{rmk}

We recall the basic convention of difference and differential algebra,
that we will frequently use in what follows:
\begin{quote}
\it
Algebraic attributes always refer to the underlying ring, ideal or algebra and the operator
prefix highlights the compatibility of the algebraic  attribute with the operator prefix.
\end{quote}
For example:
\begin{itemize*}
  \item
  A $\sgqp$-$\cF$-algebra $R$ is an $\cF$-algebra equipped
  with a derivation $\partial$ and $\sgq$
  that extend to $R$ the action of $\partial$ and $\sgq$ on the $\sgqp$-field $\cF$.
  \item
  A $\sgq$-ideal (resp. $\sgqp$-ideal)
  of a $\sgq$-ring (resp $\sgqp$-ring) $R$ is an ideal of $R$ that is set-wise invariant by $\sgq$ (resp. by $\sgq$ and $\partial$).
\end{itemize*}
Since we will use more sophisticated notions of differential algebra, we have added a section on this topic in the appendix.
Classical references for further readings are for instance \cite{kolchin-diff-algebra}, \cite{Cohn:difference}, \cite{Levin}.

\smallskip
From now on we suppose that $(\cF,\sgq,\partial)$
is a $q$-difference differential field of characteristic zero
and that we are given a $q$-difference system
of the form \eqref{eq:sys}.


\subsection{Parameterized Picard-Vessiot rings}
\label{subsec:abstractsol}

In this section we recall some results on the existence of abstract solutions for linear $q$-difference systems endowed with the action of a derivation.
For a quick survey of the needed notions of differential algebra,
we refer to Appendix \ref{app:diffalgebra}.

\begin{defn}\label{defn:PPV}
A $\sgqp$-$\cF$-algebra $R$ is a {\em  parameterized Picard-Vessiot ring } for  \eqref{eq:sys} if
\begin{enumerate*}
\item $R$ is a simple $\sgqp$-$\cF$-algebra, \ie ,  there are no non-trivial ideal    being set-wise fixed by $\sgq$ and $\partial$;
\item there exists a $Z\in \GL_\nu(R)$ such that $\sgq(Z)=AZ$;
\item $R = \mathcal{F}\{Z, \det Z^{-1}\}_\partial$, i.e.,
$R$ is generated as a $\cF$-algebra by the entries of $Z$, the inverse of the determinant of $Z$ and all their derivatives with respect to $\partial$
(see Definition \ref{defn:differential-polynomials}).
\end{enumerate*}
\end{defn}

\noindent
Such a ring always exists: We consider the ring of differential
polynomials $S =\cF \{ Y,\det Y^{-1}\}_{\partial}$,
where $Y$ is a matrix of differential indeterminates over $\cF$ of order $\nu$.
A $q$-difference operator compatible with the differential structure
of the $\partial$-$\cF$-algebra $S$ can be defined setting:
$$
\begin{array}{l}
 \sgq(Y) =A Y,\\
  \sgq(\partial Y)= \partial(\sgq Y)=A \partial(Y) +\partial(A) Y,\\
  \hbox{and so on, using the commutativity of $\sgq$ and $\partial$ and the Leibnitz rule.}
\end{array}
$$
Any quotient
of the ring $S$ by a maximal $\sgqp$-ideal, i.e.,
a maximal ideal in the set of $\sgqp$-ideals, is a $\sgqp$-Picard-Vessiot ring.

\begin{lemma}[{\rmfamily\cite[Lemma 6.8]{HardouinSinger}}]
Any simple $\sgqp$-$\cF$-algebra $R$ that is finitely
generated as $\partial$-$\cF$-algebra,
has the following structure: {There exist a positive integer $t$ and $e_0,\dots,e_{t-1}$ idempotents of $R$,
generating the ideals $R_i:=e_iR$,
such that:
\begin{enumerate*}
\item
$R=R_0\oplus\dots \oplus R_{t-1}$;
\item
$\sgq$ permutes transitively the set $\{R_0,\dots,R_{t-1}\}$
and $\sgq^t$ leaves each $R_i$ invariant;
\item
each $R_i$ is a domain and a simple $(\sgq^t,\partial)$-$\cF$-algebra.
\end{enumerate*}}
\end{lemma}

A $\sgqp$-$\cF$-algebra
satisfying the three properties of the lemma above
is called a $\sgqp$-$\cF$-pseudo-domain, by analogy with the definition of
$\sgq$-$\cF$-pseudo-domain. See \cite[\S1.1]{Wibchev}.
We immediately obtain:

\begin{cor}
A parameterized Picard-Vessiot ring is a $\sgqp$-$\cF$-pseudo domain.
\end{cor}

Since $\sgq$ and $\partial$ commute, the field  $\cC$ is naturally a $\partial$-field. If it is $\partial$-closed (see
Definition \ref{defn:differentially-closed-field}) we have:

\begin{prop}[{\rmfamily\cite[Proposition 2.4]{HardouinSinger}}]
\label{prop:PVsigmaclosed}
If $\cC$ is a $\partial$-closed field, then the $\sgq$-constants of any parameterized Picard-Vessiot
ring coincide with $\cC$.
Moreover, any two parameterized Picard-Vessiot rings are isomorphic as $\sgqp$-$\cF$-algebras.
\end{prop}

In analogy with \cite[Definition 2.1]{ChatHardouinSinger} we set:

\begin{defn}\label{defn:weak parameterized Picard-Vessiot ring}
A $\sgqp$-$\cF$-algebra $R$ is a weak parameterized Picard-Vessiot ring
for \eqref{eq:sys}  if
\begin{enumerate*}
\item $R =\cF\{Z,\det Z^{-1}\}_\partial$, where $Z \in \GL_\nu(R)$ and $\sgq(Z)=AZ$;
\item $\cC:=\cF^{\sgq}=R^{\sgq}$.
\end{enumerate*}
\end{defn}

Proposition \ref{prop:PVsigmaclosed} says that,
if we have a differentially closed field of $\sgq$-constants,
a parameterized Picard-Vessiot ring
is always a weak parameterized Picard-Vessiot ring.
However M. Wibmer has proved that assuming that $\cC$ is
is algebraically closed is enough to ensure the existence:

\begin{prop}[{\cite[Corollary 9]{Wibdesc} and \cite[Proposition 1.16 and Corollary 1.19]{diviziohardouinPacific})}]
\label{prop:weakPVexistence}
If $\cC$ is algebraically closed, there exists a weak parameterized
Picard-Vessiot ring $R$, which is moreover $\sgq$-simple, \ie , has no non-trivial $\sgq$-ideals.
\end{prop}

\begin{rmk}
\lucia{Notice that the proposition above proves more than needed. Indeed the
weak parameterized
Picard-Vessiot ring constructed in Proposition \ref{prop:weakPVexistence} is
$\sgq$-simple, hence it is \emph{a fortiori}
a simple $\sgqp$-$\cF$-algebra. This means that it is also a parameterized
Picard-Vessiot ring in the sense of Definition \ref{defn:PPV}.
While it is relatively easy to construct a $\sgqp$-simple
parameterized Picard-Vessiot ring, Wibmer's idea for the construction of
 $\sgq$-simple parameterized Picard-Vessiot ring is quite subtle.}
\end{rmk}

If $\cC$ is algebraically closed, uniqueness  is assured only after extension  to a differential closure of $\cC$.
See Proposition \ref{prop:PVsigmaclosed} above and \cite[page 164]{Wibdesc}.

\medskip
We can now define the parameterized difference Galois group:

\begin{defn}\label{defn:functorautomorphisms}
Let $R$ be a weak parameterized Picard-Vessiot ring for a $q$-difference system \eqref{eq:sys} defined over $\cF$.
We define the parameterized difference Galois group of $R/\cF$ as {the covariant functor:}
$$
\begin{array}{cccc}
 G_R^\partial : & \partial\mbox{-}\cC\mbox{-algebras}  & \rightarrow &  \mbox{Groups}  \\
    & \cS & \mapsto & \Aut_{\cF\otimes \cS}^{\sgqp}(R \otimes_\cC\cS),
\end{array}
$$
where: the derivation $\partial$ of $R \otimes_\cC\cS$ is  { defined by $\partial_R \otimes id + id \otimes \partial_\cS$;}
the operator $\sgq$ extends from $R$ to $R \otimes_\cC\cS$ by linearity, acting trivially over $\cS$;
the notation $\Aut_{\cF\otimes \cS}^{\sgqp}(R \otimes_\cC\cS)$ stands for the group of $\sgqp$-$\cF \otimes_\cC\cS$-automorphisms of $R \otimes_\cC\cS$.
\end{defn}

We refer to Appendix \ref{appendix:diffgeometry} for the geometric definition used in the following proposition:

\begin{prop}\label{prop:groups}
The parameterized difference Galois group $G_R^\partial$ is representable
by a $\partial$-$\cC$-subgroup scheme   of $\mathrm{GL}_\nu(\cC)$.
\end{prop}

\begin{proof}
We omit this proof which is a straightforward parameterized \charl{analog} of \cite[Proposition 2.2]{ChatHardouinSinger}.
\end{proof}

\begin{rmk}\label{rmk:classicalPV}
Let us suppose that $(\cF,\sgq,\partial)$ is a $q$-difference field with a trivial derivation $\partial$ and let us consider the field of
differential polynomials $S=\cF\{Y,\det Y^{-1}\}_\partial$
with the $q$-difference structure induced by \eqref{eq:sys}.
By construction, the derivation on $S$ is non-trivial, indeed
$\partial Y$ is non-zero. The $\partial$-ideal generated by $\partial Y$ is a $\sgqp$-ideal, in fact we have:
$$
\sgq(\partial Y)=\partial(A)Y+A(\partial Y)=A(\partial Y) \mbox{~and hence~} \sgq(\partial^k Y)= A (\partial^k Y) \mbox{~for all~} k \in \N.
$$
Moreover the quotient $S/(\partial Y)$ is nothing else that the ring of polynomials
$\cF [Y, \det Y^{-1} ]$, endowed with the $q$-difference structure given by $\sgq(Y)=AY$ and the trivial derivation.
Any of its maximal $\sgq$-ideals is the quotient of a maximal $\sgqp$-ideal of $S$ by $(\partial Y)$.
The reader familiar with the Galois theory of difference equations
will have already noticed that
the parameterized Picard-Vessiot ring that we have
constructed in this way is actually a usual Picard-Vessiot
ring for \eqref{eq:sys}
over $\cF$ (see \cite{vdPutSingerDifference}).
In this sense, we say that, when $\partial$ is
the trivial derivation, {a} (weak) parameterized Picard-Vessiot ring allows
to recover {a}
usual (weak) Picard-Vessiot ring. \lucia{Then the} definition above boils down to the definition of the usual difference Galois group,
which is representable by a linear algebraic group.
\par
This is true in a more general sense. For a general derivation $\partial$
and for a given (weak) parameterized Picard-Vessiot ring
$R=\cF\{Z,\det Z^{-1}\}_\partial$ for
\eqref{eq:sys} over $\cF$, we can consider the subalgebra
{$R_0=\cF[Z, \det Z^{-1}]$} of $R$. This is the usual (weak) Picard-Vessiot ring for
\eqref{eq:sys} over $\cF$  (see \cite[Proposition 2.8]{HardouinSinger}).
{ We remind that the  difference Galois group $G_R$ of $R_0/\cF$ is defined as follows
\beq\label{eq:G_R}
\begin{array}{cccc}
 G_R : & \cC\mbox{-algebras}  & \rightarrow &  \mbox{Groups}  \\
    & \cS & \mapsto & \Aut_{\cF\otimes \cS}^{\sgq}(R_0\otimes_\cC\cS),
\end{array}
\eeq
where
the operator $\sgq$ extends from $R_0$ to $R_0 \otimes_\cC\cS$ by linearity, acting trivially over $\cS$.
The difference Galois group $G_R$
is representable by a group-scheme defined over $\cC$ (see \cite[Prop. 2.2]{ChatHardouinSinger}).}
\par
{Notice that $G_R$ is an abuse of notation and we should write $G_{R_0}$ instead.
We prefer not to use a complicate notation, since there will be no confusion in the text below. }
\end{rmk}

\subsection{Weak parameterized Picard-Vessiot rings associated with meromorphic solutions}
\label{subsec:merosol}

We assume that our base $q$-difference field $(\cF,\sgq)$ is a subfield of the field $\Mer{\C^\ast}$ of meromorphic functions over $\C^\ast$ and
that $|q|\neq 0,1$.

\begin{rmk}\label{rmk:AssumptionsOnTheBaseField}
This is a restrictive assumption, but not as much as one could imagine.
Later on we will focus on a $q$-difference field  $(K(x),\sgq)$ of rational functions, where $K$ is a finitely generated extension of $\Q$.
Of course we can always embed $K$ into $\C$. If $q$ is transcendental over $\Q$, then we can always choose
an embedding such that $q$ will have an image in $\C$ of norm different than $1$. Of course, if $q$ is algebraic this is possible ``most of the times'' but not always.
See also Remark \ref{rmk:p-adic}.
\end{rmk}

We consider the elliptic curve $E:=\mathbb{C}^*/q ^{\mathbb{Z}}$
and its field $\C_E$ of elliptic functions, that is the meromorphic functions over $\C^\ast$ that are invariant by $\sgq$.
We recall the following result:

\begin{thm}[{\cite[Theorem 3]{Praag2}}]\label{thm:praagman}
Any linear $q$-difference system $\sgq(\vec{y})=A\vec{y}$, with $A\in\GL_\nu(\Mer{\C^\ast})$,
has a \lucia{a basis of solutions with coefficients}
in  $\Mer{\C^\ast}$, linearly independent over $\C_E$.
\end{thm}

The theorem above requires some comments. By a full basis of linearly independent solutions we mean $\nu$ solution vectors
$\vec{y}_1,\dots,\vec{y}_\nu\in\Mer{\C^\ast}$,
linearly independent over $\C_E$. One usually say that
$\sgq(\vec{y})=A\vec{y}$ admits a \emph{fundamental matrix} of solutions $Y\in\GL_\nu(\Mer{\C^\ast})$, whose columns are $\vec{y}_1,\dots,\vec{y}_\nu\in\Mer{\C^\ast}$, which summarize
the conclusion of the theorem.
To the best of our knowledge, \lucia{there is no constructive proof
of the existence
of a basis of} meromorphic solutions
of a general $q$-difference system with meromorphic coefficients.
We are able to do it in full generality under the assumption that $A\in\GL_\nu(\C(x))$
(see \cite{DreyfusMeromorphicSolutions}).

\begin{exa}
Let us assume that $|q|>1$.
The Jacobi theta function
$\Theta_q(x)=\sum_{n\in\Z}q^{-n(n-1)/2}x^n $
is an element of $\Mer{\C^*}$.
It is solution of the $q$-difference equation
$y(qx)=qx\,y(x)$.
Following \cite{Sfourier}, one can use the following meromorphic functions
\begin{itemize*}
\item
$\Theta_q(cx)/\Theta_q(x)$, with $c\in \C^*$,  solution of $y(qx)=cy(x)$,
\item
$x\Theta_q^\p(x)/\Theta_q(x)$, solution of $y(qx)=y(x)+1$,
\end{itemize*}
to write a meromorphic fundamental solution to any $q$-difference system that is regular singular system at $0$ or at $\infty$ (see \cite[\S 0.1]{Sfourier}).
\end{exa}

\begin{rmk}\label{rmk:p-adic}
If $K$ is a finitely \charl{generated} extension of $\Q$ and $q$ is not a root of unity,
one can always embed $K$ in $\C$ or in $\C_p$ in a way that
$|q| \neq 1$. Let us focus on the case of an embedding in $\C_p$.
Since the Jacobi Theta function converges over $\C_p^\ast$,
one can transpose the results of J. Sauloy and T. Dreyfus to the $p$-adic setting
and construct a fundamental matrix of solutions of
$p$-adic meromorphic function over $\C_p^\ast$
for a linear $q$-difference system defined over $K(x)$.
In other words, one can always assume that a system with coefficients
in $K(x)$ has a \charl{ fundamental matrix of solutions with  meromorphic coefficients} in some sense, archimedean or $p$-adic.
This result is commonly accepted but yet there are no references for it.
It would allow us to apply the results below to a broader range of cases.
\end{rmk}

Theorem \ref{thm:praagman} by Praagman ensures the existence of nice solutions, but it has a cost.
The field of meromorphic functions over $\C^*$ fixed by $\sgq$ coincides with the
field of meromorphic functions over the torus $\C^\ast/q^\Z$, therefore we have enlarged considerably the
field of constants with respect to the ``expected'' algebraically closed
field of constants $\C$.
\par
Since  $\sgq$ and $\partial:=\frac{d}{dx}$ commute, the derivation $\partial$ stabilizes $\C_E$ inside $\Mer{\C^*}$,
so that $\C_E$ is naturally endowed with a structure of $\partial$-field.
It is a trivial $\sgq$-field.
Let $\overline{\C}_E$ be an algebraic closure of $C_E$  and $\wtilde\C_E$ a differential closure
of $\overline\C_E$ with
respect to $\partial$
(\cf \cite[\S9.1]{cassisinger}). We still denote by $\partial$ the derivation of $\wtilde\C_E $.
One  endows $\wtilde\C_E$ and $\overline{\C}_E$ with a structure of trivial $\sgq$-field.
{We want to show} that  these trivial $\sgq$-fields are  \emph{compatible} with the $\sgq$-field $\C(x)$ in the sense that  there exists a  joint $\sgq$-field extension of   $\C(x)$ and  $\widetilde \C_E$ {(see Corollary \ref{cor:constantextension} below).}

\begin{lemma}\label{lemma:lindisjoint}
Let $\mathcal{C}$ be a field extension of $\C$ endowed with a trivial action {of $\sgq$}.
The fields $\mathcal{C}$ and $\C(x)$ are linearly disjoint \lucia{in $\cC(x)$} over $\C$.
\end{lemma}

\begin{rmk}
\lucia{The inclusion $\C\hookrightarrow\cC$ extends to an inclusion of
field of rational functions  $\C(x)\hookrightarrow\cC(x)$,
therefore the statement above makes sense.}
\end{rmk}

\begin{proof}
This is a well known property of difference fields and the proof uses very standard ideas. We give it here for completeness.
Let $f_0,\dots, f_r\in \mathcal{C}$ be linearly independent over $\C$ and let us suppose that they become linearly dependent over $\C(x)$.
We suppose that $r>0$ is minimal for this property. Then there exist $a_1,\dots,a_r\in\C(x)\smallsetminus\{0\}$, not all belonging to $\C$,
such that
$f_0+a_1f_1+\dots+a_rf_r=0$. Applying $\sgq$ and subtracting the obtained equation, we deduce that
$(\sgq(a_1)-a_1)f_1+\dots+(\sgq(a_r)-a_r)f_r=0$. The minimality of $r$ implies that the $a_i$'s are in $\C$, against the assumption. This proves the claim.
\end{proof}

Let $\mathcal{C}$ be a trivial $\sgq$-field extension of $\C$.
Thanks to the previous lemma, we know the compositum $\mathcal{C}(x)$ of $\mathcal{C}$ and  $\C(x)$ over $\C$
coincides with the field of fractions of $\cC\otimes_\C \C(x)$.
We have:

\begin{cor}\label{cor:constantextension}
Let $\mathcal{C}$ be a $\sgqp$-field extension of $\C$ {and a trivial $\sgq$-field.} The field $\mathcal{C} (x)$   is
a $\sgqp$-field with the action of $\sgq$ defined by the properties that ${\sgq}_{\vert\mathcal{C}}$ is the
identity of $\mathcal{C}$ and $\sgq(x)=qx$.
\end{cor}

\begin{proof}
The Leibnitz rule allows to extend $\partial$ to the $\mathcal{C}\otimes_\C \C(x)$  and one can easily extend $\partial$ to its quotient (\cite[Exercices 1.5]{vdPutSingerDifferential}). The action of $\sgq$ on  {$\mathcal{C}\otimes_\C \C(x)$} defined by $id \otimes \sgq$ is injective and extends to  $\mathcal{C} (x)$. The commutativity of $\sgq$ and $\partial$ is straightforward.
\end{proof}

Corollary \ref{cor:constantextension},
taking $\mathcal{C}=\C_E$,
allows to consider the $\sgqp$-field extensions $\C_E(x)$ of $\C(x)$.
We can finally construct a weak parameterized Picard-Vessiot ring associated with Praagman's meromorphic solutions:

\begin{prop}
\label{prop:diffPV}
Let $\sgq(\vec{y})=A\vec y$,
with $A \in \GL_\nu(\C(x))$, be a $q$-difference system
and let $U \in \GL_\nu(\Mer{\C^*})$ be a fundamental solution matrix.
The ring $R_E :=\C_E(x)\{ U,\det U^{-1}\}_\partial$
is a weak parameterized Picard-Vessiot ring over
$\C_E (x)$ for $\sgq(\vec{y})=A\vec y$ and is an integral domain.
\end{prop}

\begin{proof}
It is enough to notice that $R_E\subset \Mer{\C^*}$ and that $\C_E\subset R_E^{\sgq}\subset \Mer{\C^*}^{\sgq}=\C_E$.
\end{proof}

\subsection{Properties of the weak Picard-Vessiot rings $R$, $R_E$ and $\wtilde R$}
\label{subsec:summaryPV}

Let $\sgq(\vec{y})=A\vec y$,
with $A \in \GL_\nu(\C(x))$, be a $q$-difference system.
We have constructed three weak parameterized Picard-Vessiot rings for $\sgq(\vec{y})=A\vec y$:
\begin{enumerate}
  \item
  the weak parameterized Picard-Vessiot ring $R$ over \lucia{$\C(x)$},
  which is $\sgq$-simple and satisfies $R^{\sgq}=\C$, constructed applying Proposition \ref{prop:weakPVexistence}
  to $\sgq(\vec{y})=A\vec y$, seen as a system defined over $\cF=\C(x)$;
  \item
  the  weak parameterized Picard-Vessiot ring $R_E$ over \lucia{$\C_E(x)$}, constructed in Proposition \ref{prop:diffPV};
  \item
  the (weak) parameterized Picard-Vessiot ring $\wtilde R$ over \lucia{$\wtilde\C_E(x)$}, constructed applying
  Proposition \ref{prop:PVsigmaclosed} to $\sgq(\vec{y})=A\vec y$,
  seen as a system defined over $\cF=\wtilde\C_E(x)$.
\end{enumerate}

We remind that $R$ can be written in the form
\begin{equation}\label{eq:PVform}
R =\C(x)\{Y,\det Y^{-1}\}_\partial/ \mathfrak{q},
\end{equation}
where $Y$ is an invertible matrix satisfying
the system $\sgq(\vec{y}):=A\vec{y}$ and
$\mathfrak{q}$ is not only a maximal $\sgqp$-ideal but also a maximal $\sgq$-ideal (since $R$ is $\sgq$-simple, after Proposition \ref{prop:weakPVexistence}).

Definition \ref{defn:functorautomorphisms} applied to the three settings above allows to define the group schemes
\lucia{$G^\partial_{R}$, $G^\partial_{R_E}$ and $G^\partial_{\wtilde R}$, respectively.}
As functors they are represented by
\lucia{a $\partial$-$\C$-subgroup scheme of $\GL_\nu(\C)$
(resp.
a $\partial$-$\C_E$-subgroup scheme of $\GL_\nu(\C_E)$,
a $\partial$-$\wtilde\C_E$-subgroup scheme of $\GL_\nu(\wtilde\C_E)$).}
See Proposition \ref{prop:groups}.
We will prove that they become
isomorphic after a convenient field extension. To do so, we need to prove some properties of the three
Picard-Vessiot rings above and to give a Tannakian description of each one of the three groups
\lucia{$G^\partial_{R}$, $G^\partial_{R_E}$ and $G^\partial_{\wtilde R}$.}
It will be the object of
\S\ref{sec:qdiffmodules} and \S\ref{sec:Galoisgroups}.

\smallskip
The following statement is differential analogue of
\cite[Proposition 2.4]{ChatHardouinSinger}.

\begin{prop}\label{prop:pvextell}
Let $\sgq(\vec{y})=A\vec{y}$
be a $q$-difference system defined over $\C(x)$.
Let $\cF$ be a $\sgqp$-field extension of $\C(x)$ of the form $\cC(x)$,
where $\cC$ a $\sgqp$-field extension of $\C$, which is a trivial $\sgq$-field
(for instance $\cF=\C_E(x)$ or $\wtilde\C_E(x)$).
In the notation of \lucia{Eq. \eqref{eq:PVform}},
$S := \cF\{Y, \det Y^{-1} \}_\partial/ \mathfrak{q}\cF$ is a  parameterized
Picard-Vessiot ring for $\sgq(\vec{y})=A\vec{y}$ \lucia{over $\cF$} and $S^{\sgq} = \cC$.
\end{prop}

\begin{proof}
First we remark that  $\mathfrak{q} \cF \subsetneq \cF\{Y, \det Y^{-1} \}_\partial$ and hence that $S$ is non-zero.
Indeed, if $1 =\sum_{i \in I} \lambda_i P_i$ with $P_i \in \mathfrak{q}$ and  $\lambda_i \in \cF$, it is enough to expand the $\lambda_i$'s in a $\C(x)$-basis of
{$\cF$} to conclude that  $1 \in \mathfrak{q}$.
\par
We consider the natural map of $\sgq$-rings
$$
\phi: R \otimes_\C\cC \rightarrow S.
$$
We want to prove that $\phi$ is injective.
Since $\phi(1\otimes 1)=1$ and $S\neq 0$, $\mathfrak{I}:=\mathop{\rm Ker}\phi$ is a proper $\sgq$-ideal of $R \otimes_\C\cC $.
Proposition \ref{prop:weakPVexistence} implies that $R$ is $\sgq$-simple.
Moreover $R^{\sgq}=\C$.
Therefore the $\sgq$-ideal
$\mathfrak{I}$ in $R \otimes_\C\cC$ is generated by $\mathfrak{I} \cap R$
(see  \cite[Lemma 1.11]{vdPutSingerDifference}).
We deduce that $\mathfrak{I}$ is $\{0\}$.
This means that $\phi$ is injective.
Notice that the same argument shows that
{ any $\sgq$-ideal $\mathfrak{J}$ of $R\otimes_\C\cC$ is generated by its intersection with $R$. Since
 $R$ is $\sgq$-simple, we deduce that $R\otimes_\C\cC$ is $\sgq$-simple.}
\par
Now let $R'=\phi (R \otimes \cC)$. Since $\phi$ is a $\sgq$-morphism, the ring
$R'$ is $\sgq$-simple.
Lemma \ref{lemma:lindisjoint} implies that
the field $\cF$ is the fraction field of $\cC \otimes_\C \C(x)$. Then, it is easily seen that for any
$P \in \cF\{Y, \det Y^{-1} \}_\partial$ there exist $a \in  \left(\cC \otimes_\C \C(x) \right)^\ast$ such that $aP \in \cC \otimes_\C \C(x)\{Y, \det Y^{-1} \}_\partial$.
Then,
 for all $x \in S$, there exists
$a \in  \left(\cC \otimes_\C \C(x) \right)^\ast$ such that $ax \in R'$. This proves that any $\sgq$-ideal
$\mathfrak{J}$ in $S$ is generated by $\mathfrak{J} \cap R'$,
hence $S$ is $\sgq$-simple and thus $\sgqp$-simple. We conclude that $S$ is a parameterized Picard-Vessiot ring for $\sgq(\vec{y})=A\vec{y}$.
\par
Finally, for any $c \in S^{\sgq}$, the set $\mathfrak{h}=\{a \in R' | ac \in R' \}$ is a non-zero
$\sgq$-ideal of $R'$. Since $R'$  is $\sgq$-simple, we get that $1 \in \mathfrak{h}$.
Therefore $c \in R'$ and
\lucia{$S^\sgq=(R^\p)^\sgq=\phi(R\otimes_\C\cC)^\sgq$.}
If one considers a $\C$-basis of $\cC$, which is formed by $\sgq$-constants, one can easily prove that the $\sgq$-constants of $R\otimes_\C\cC$ coincide with $R^\sgq \otimes_\C \cC= \cC$ since $R^\sgq=\C$.
\end{proof}

As corollary of the previous proposition, {we find:}

\begin{cor}\label{cor:isoppv}
{Let $\sgq(\vec{y})=A\vec{y}$
be a $q$-difference system defined over $\C(x)$ and
let $R,R_E$ and $\wtilde R$
be the weak parameterized Picard-Vessiot rings attached to
$\sgq(\vec{y})=A\vec{y}$.
As above, we write $R =\C(x)\{Y, \det Y^{-1} \}_\partial/ \mathfrak{q}$.
We consider the two rings:
$$
S:= \C_E(x)\{Y, \det Y^{-1} \}_\partial/
\mathfrak{q}\C_E(x).
\hbox{~~~and~~~}
\wtilde S := \wtilde \C_E(x)\{Y, \det Y^{-1} \}_\partial/ \mathfrak{q} \wtilde \C_E (x)
$$
Then the two natural maps
\lucia{$$
S \otimes \wtilde \C_E\longrightarrow R_E \otimes \wtilde\C_E
\hbox{~~~and~~~}
\wtilde S\longrightarrow\wtilde R
$$}
are both an isomorphism of $\sgqp$-$\wtilde\C_E(x)$-algebras.}
\end{cor}

\begin{proof}
By Proposition \ref{prop:pvextell}, applied to $\cF=\C_E(x)$ and $\cF=\wtilde \C_E(x)$, we find
that $S$ (resp. $\wtilde  S$) is a parameterized Picard-Vessiot ring
{for $\sgq(\vec{y})=A\vec{y}$  over \lucia{$\C_E(x)$} (resp. $\wtilde\C_E(x)$)}
such that $S^{\sgq} =\C_E$ (resp. $\wtilde S^{\sgq} =\wtilde \C_E$).
Since $\wtilde \C_E$ is differentially closed, \lucia{Proposition \ref{prop:PVsigmaclosed}}
assures that two parameterized Picard-Vessiot rings for the same $q$-difference system over
$\wtilde \C_E(x)$ are isomorphic as  $\wtilde \C_E(x)$-$\sgqp$-algebras. The first isomorphism
follows from this fact.
\par
The second isomorphism comes from a parameterized version of
\cite[Proposition 2.7]{ChatHardouinSinger}.  Its proof follows line by line the proof in the algebraic case, but we give it here for sake of completeness. Let us denote by $F_E$ the fraction field
of $R_E$ and let $X=(X_{i,j})$ be a $\nu \times \nu$-matrix
of differential indeterminates over $F_E$. Let $\mathcal S :=\C_E(x) \{X, \det X^{-1} \}_\partial
\subset F_E \{X, \det X^{-1} \}_\partial $. Define a $\sgqp$-structure on $F_E \{X, \det X^{-1} \}_\partial $ by setting $\sgq(X):=AX, \sgq(\partial X):= A \partial X + \partial A X$, and so on.
This induces a $\sgqp$-structure on $\mathcal S$. Since $S$ is a parameterized Picard-Vessiot ring
for $\sgq(Y)=AY$ view over $\C_E(x)$, we can write $S =\mathcal{S}/ \mathfrak{p}$, where
$\mathfrak{p}$ is a maximal $\sgqp$-ideal of $\mathcal{S}$. Now, let $U \in \GL_\nu(R_E)$ be fundamental
solution matrix of $\sgq(Y)=AY$. Define \lucia{$Y = (Y_{i,j}) \in \GL_\nu(F_E\{X, \det X^{-1} \}_\partial)$}
via $Y:= U^{-1} X$ and remark that $\sgq(Y)=Y$ and $F_E\{X, \det X^{-1} \}_\partial=F_E\{Y, \det Y^{-1} \}_\partial$. Define $\mathcal{S}_1 := \C_E\{Y, \det Y^{-1} \}_\partial$. The ideal
$\mathfrak{p} \subset \mathcal{S} \subset F_E\{X, \det X^{-1} \}_\partial$ generates a
$\sgqp$-ideal $(\mathfrak{p})$ in $F_E\{X, \det X^{-1} \}_\partial$, which intersected with $\mathcal{S}_1$ gives a $\partial$-ideal $\mathfrak{a}$. Since $\wtilde \C_E$ is differentially closed
and $\mathcal{S}_1/\mathfrak{a}$ is differentially finitely generated over $\C_E$, we find a differential
homomorphism \lucia{$\mathcal{S}_1  \rightarrow \mathcal{S}_1/\mathfrak{a} \rightarrow \wtilde\C_E $.}
We can extend this homomorphism into a $\sgqp$-morphism
\lucia{$F_E\{X, \det X^{-1} \}_\partial  =F_E \otimes \mathcal{S}_1 \rightarrow F_E \otimes_{\C_E} \wtilde\C_E $.}
Restricted to $\mathcal{S}$, we find a $\sgqp$-morphism $\mathcal{S} \rightarrow
F_E \otimes \wtilde \C_E$, whose kernel contains $\mathfrak{p}$. By maximality of
$\mathfrak{p}$, we have equality and we get an embedding $\iota : S =\mathcal{S}/\mathfrak{p} \rightarrow F_E \otimes \wtilde \C_E$.
Now, if we denote by \lucia{$V \in \GL_\nu(S)$} a fundamental solution matrix
of $\sgq(Y)=AY$ {and we remind that $(F_E\otimes \wtilde \C_E)^{\sgq}=\wtilde \C_E$, we find}
that $\iota(V)= U D$ with \lucia{$D \in \GL_\nu(\wtilde \C_E)$.}
Since $S$ (resp. $R_E$)
is differentially generated over $\C_E(x)$ (resp. $\wtilde \C_E(x)$) by $V$ (resp $U$) and the inverse of its determinant,   this allows us
to conclude that $\iota(S \otimes \wtilde \C_E)= R_E \otimes \wtilde \C_E$.
\end{proof}


\section{The category of $q$-difference modules}
\label{sec:qdiffmodules}

We are interested in giving an interpretation of Picard-Vessiot extensions from a categorical point of view, therefore we introduce here
the category of $q$-difference modules.
Since we are interested in studying the action of the derivation $\partial$, we will quickly review the basic definitions
and properties of differential Tannakian categories, introduced by
Kamensky  \cite{kamtan}
and Ovchinnikov  in  \cite{DifftanOv}.

\subsection{$q$-difference modules}

Let $(\cF,\sgq,\partial)$ be a $\sgqp$-field \charl{of characteristic zero} and let $\cC=\cF^{\sgq}$.

\begin{defn}
A {$q$-difference module $\cM_\cF=(M_\cF,\Sgq)$ (of rank $\nu$) over $\cF$}
\index{$q$-difference!module}
is a finite dimensional $\cF$-vector space $M_{\cF}$ (of dimension $\nu$)
equipped with an invertible $\sgq$-semi-linear operator $\Sgq:M_\cF\to M_\cF$, \ie,
a bijective additive map from $M_\cF$ to itself such that
$$
\Sgq(fm)=\sgq(f)\Sgq(m),
\hbox{~for any $f\in \cF$ and $m\in M_{\cF}$.}
$$
We will call $\Sgq$ a $q$-difference operator over $M_\cF$ or the $q$-difference operator of $\cM_\cF$.
\par
A $q$-difference submodule $\cN_\cF$ of $\cM_\cF$ is a $\cF$-vector subspace $N_\cF$
of $M_\cF$ that is set-wise invariant with respect to $\Sgq$. Then, $\cN_\cF= (N_\cF,\Sgq{|_{N_{\cF}}})$ is
a $q$-difference module.
\par
A {morphism of $q$-difference modules (over $\cF$)} is a morphism
of the underlying $\cF$-vector spaces, commuting with the $q$-difference operators defined on the domain and on the image of the morphism.
We denote by $\Diff(\cF, \sgq)$ the category of $q$-difference modules over $\cF$.
\end{defn}

Let $\cM_{\cF}=(M_{\cF},\Sgq)$ be a $q$-difference module over $\cF$ of rank $\nu$.
We fix a basis $\ul e$ of $M_{\cF}$ over $\cF$.
Let $A\in \GL_\nu(\cF)$ be such that:
$$
\Sgq\ul e=\ul e A.
$$
If $\ul f$ is another basis of $M_\cF$, such that $\ul f=\ul e F$, with $F\in\GL_\nu(\cF)$, then
$\Sgq\ul f=\ul f B$, with $B=F^{-1}A\sgq(F)$. Conversely, given an invertible matrix $A \in \GL_\nu(\cF)$, one construct
a $q$-difference module $\cM_\cF$ as follows: $M_\cF= \cF^\nu$ and $\Sgq\ul e =\ul e A$ with $\ul e$ the canonical basis of $\cF^\nu$.

\smallskip
The elements $m\in M_\cF$ such that $\Sgq(m)=m$ are called horizontal.
If a horizontal element $m$ corresponds to a vector $\vec{y}\in \cF^\nu$ with respect to the basis $\ul e$, we have:
$\ul e\vec{y}=\Sgq(\ul e \vec{y})=\ul e A\sgq(\vec{y})$.
Therefore $\vec{y}$ verifies the linear $q$-difference system $\sgq(\vec{y})=A^{-1}\vec{y}$, that we call the linear difference system
associated to $\cM_\cF$ with respect to the basis $\ul e$.

\smallskip
The constructions of linear algebra
(i.e., direct sums, duals, the tensor products) of the underlying vector spaces of two  $q$-difference modules over $\cF$ can be endowed with a structure  of $q$-difference modules (see for instance  \cite[Chapter 12]{vdPutSingerDifference},
\cite[Part I]{DVInv} or \cite{gazette}). The category $\Diff(\cF, \sgq)$ is a tensor category and we denote by
$\mathbf 1 =(\cF,\sgq)$ the unit object for the tensor product.  It is also is a rigid
category, \ie, it has internal Homs and each object is canonically isomorphic to its bidual. It is therefore a Tannakian category in the sense of
\cite{delmil} (see \cite{Sauloy-Galois-AENS}).
If $\cC$ is algebraically closed, the general theory of Tannakian categories ensures
that it is equivalent to the category of representation of
a certain {$\cC$}-group-scheme $G$.

\subsection{The differential Tannakian structure of $\Diff(\cF, \sgq)$}

In this section we define the prolongation functor in the general framework of projective modules.
The definition may seem very abstract at the first glance but
we will show in Example \ref{exa:prolongation}
that it is an incarnation of Remark \ref{rmk:prolongation}.

\smallskip
We consider a $\partial$-field $k$ and a $\partial$-$k$-algebra $\cS$.
We denote  by $\cS[\partial]_{\leq 1}$
the 2-dimensional free $\cS$-module of differential operators of
order less or equal to $1$.
In agreement with the Leibniz rule, the
right $\cS$-module structure of $\cS[\partial]$  is given by $\partial. a = a.\partial + \partial(a)$.

\begin{defn}
We define on the category $\Proj_\cS$
of finitely generated projective modules over $\cS$ an endofunctor  $F_\partial$,
called prolongation functor,
as follows:
\begin{itemize*}
\item
For $M$ an object of $\Proj_\cS$, we define $F_\partial(M):=\cS[\partial]_{\leq 1}\otimes_\cS M$,
where the tensor product is consider with respect to the right $\cS$-module structure of $\cS[\partial]_{\leq 1}$.
The $\cS$-module structure of $F_\partial(M)$ is defined by:
$\lambda(\partial\otimes v)= \partial\otimes \lambda v -\partial(\lambda)\otimes v$, for all $\lambda \in \cS$ and $v \in M$,
and extended by linearity.
\item
If $f \in\Hom_{Proj_\cS}(M,N)$,
we define $F_\partial(f): F_\partial(M) \rightarrow F_\partial(N)$ as:
$F_\partial(f)(\partial^i\otimes m)=\partial^i\otimes f(m)$, for $i=0,1$, where
we have used the convention that $\partial^0$ is the identity map.
\end{itemize*}
\end{defn}

\begin{rmk}
We will informally call constructions of linear differential algebra the family of all the constructions of linear algebra plus $F_\partial$.
Notice that, if $\partial$ is the trivial derivation, then $F_\partial(M)$
coincides with the direct sum $M\oplus M$.
\end{rmk}

{The underlying vector spaces of the objects of $\Diff(\cF,\sgq)$ form a subcategory of $\Proj_\cF$.}
Since $\cF$ is a field, $\Proj_\cF$ is the category of vector spaces over $\cF$,
that we will also denote $\Vect_\cF$.
Given an object $\cM_\cF =(M_\cF,\Sgq)$ of $\Diff(\cF,\sgq)$,
we are able to extend the action of $\Sgq$ to $F_\partial(M_\cF)$ via
$$
\Sgq(\partial^i(m)):=\partial^i(\Sgq(m)),
\hbox{~for $i=0,1$ and $m \in M_\cF$.}
$$
We set $F_\partial(\cM_\cF)=(F_\partial(M_\cF),\Sgq)$.
This
shows that $F_\partial$ {extends} to  an endofunctor  of $\Diff(\cF,\sgq)$. Together
with this additional structure,  $(\Diff(\cF,\sgq),F_\partial)$
is a differential Tannakian category over $\cC$ as defined in \cite[\S 4.4]{GilGorchOV},
\ie, a $\cC$-linear, tensor, rigid category together with a prolongation functor, satisfying some
natural commutative diagrams, that we are not recalling here.

\begin{exa}
\label{exa:prolongation}
Let $\cM_\cF =(M_\cF,\Sgq)$ be a $q$-difference module over $\cF$. We fix a basis $\ul e=(e_1,\dots,e_{\nu})$ of $M_\cF$
such that $\Sgq\ul e=\ul e A$, for some $A\in\GL_\nu(\cF)$.
A basis of $F_\partial(M_\cF)$ is given by $(\ul e, \partial\otimes\ul e)$. The definition of $\Sgq$ on
$F_\partial(M_\cF)$ is reminiscent of
Remark \ref{rmk:prolongation}:
$$
\Sgq(\ul e, \partial\otimes\ul e)=(\ul e, \partial\otimes\ul e)
\begin{pmatrix}
  A & \partial(A) \\
  0 & A
\end{pmatrix}.
$$
\end{exa}

Following \cite[Definition 4.9]{GilGorchOV}, we recall the notion of differential
fiber functor.

\begin{defn}
Let $\cS$ be a  $\partial$-$\cC$-algebra. We say that a functor
$\omega : \Diff(\cF,\sgq) \rightarrow \Proj_\cS$ is a differential fiber functor
over $\cS$ if it is exact, faithful, $\cC$-linear, tensor compatible
and if it commutes to $F_\partial$,
\ie, if $F_\partial \circ \omega =\omega \circ F_\partial$ as a natural isomorphism.
We say that $\omega$ is a neutral differential fiber functor if $\cS=\cC$.
\end{defn}

\begin{rmk} For further reference we point out that:
\begin{trivlist}
\item $\bullet$
A differential fiber functor is also a fiber functor for the classical
Tannakian theory \cite[p. 148]{delmil}.
\item $\bullet$
The forgetful functor $\eta_\cF: \Diff(\cF,\sgq) \rightarrow \Vect_\cF$, which assigns to any $q$-difference module its underlying $\cF$-vector space,  is a differential fiber functor over $\cF$.
\end{trivlist}
\end{rmk}

Since one of our main purposes is to compare distinct fiber functors, we introduce the functor of  differential tensor
morphisms between two differential fiber functors.

\begin{defn}[Def. 1.12 in \cite{delmil}]\label{defn:isomfibregroup}
Let $\omega_1,\omega_2 : \Diff(\cF,\sgq) \rightarrow \Proj_\cS$ be two differential fiber functors.
For any $\cS$-algebra $\cR$, we define $\Hom^\otimes(\omega_1,\omega_2)(\cR)$ as the set of all sequences of the form
$\l\{ \lambda_{\cX_\cF} | \cX_\cF \mbox{ object of } \Diff(\cF,\sgq)\r\}$,
such that:
\begin{itemize*}
\item
$\lambda_{\cX_\cF}$ is an $\cR$-linear homomorphism from $\omega_1(\cX_\cF) \otimes \cR$ to $\omega_2(\cX_\cF) \otimes \cR$,
\item
$\lambda_{\mathbf{1}}$ is the identity on $\mathbf{1} \otimes \cR$,
\item
for every $\alpha \in\Hom(\cX_\cF,\cY_\cF)$, we have
\charl{$\lambda_{\cY_\cF} \circ (\omega_1(\alpha) \otimes id_\cR) =(\omega_2(\alpha) \otimes id_\cR)\circ \lambda_{\cX_\cF}$},
\item  $\lambda_{\cX_\cF}\otimes \lambda_{\cY_\cF} =\lambda_{\cX_\cF\otimes \cY_\cF }$.
\end{itemize*}
For a $\partial$-$\cS$-algebra $\cR$,
we define $\Hom^{\otimes,\partial}(\omega_1,\omega_2)(\cR)$ as the
subset of $\Hom^\otimes(\omega_1,\omega_2)(\cR)$ of all sequences such that
$F_\partial(\lambda_{\cX_\cF}) =\lambda_{F_\partial(\cX_\cF )}$,
where the $F_\partial$ on the left hand side is the prolongation functor on $\Proj_\cR$ whereas the $F_\partial$ on the right hand side
is the prolongation functor in $\Diff(\cF,\sgq)$ (see \cite[\S 4.3]{DifftanOv}).
\end{defn}

The functor $\Hom^{\otimes,\partial}(\omega_1,\omega_2)$,
 composed with the forgetful functor from $\partial$-$\cS$-algebras to $\cS$-algebras is a subfunctor of $\Hom^{\otimes}(\omega_1,\omega_2)$.
By  \cite[Prop.6.6]{Delgrotfest} the functor
 $\Hom^{\otimes}(\omega_1,\omega_2)$
is representable by a $\cS$-scheme.

\smallskip
Since morphisms of tensor functors are isomorphisms (see \cite[Proposition 1.13]{delmil}),
differential morphisms of differential tensor functors are also differential isomorphisms.
Thus,  we will now write $\Isom^{\otimes,\partial}(\omega_1,\omega_2)$ (resp. $\Isom^{\otimes}(\omega_1,\omega_2)$) instead of $\Hom^{\otimes,\partial}(\omega_1,\omega_2)$ (resp. $\Hom^{\otimes}(\omega_1,\omega_2)$) and,  when $\omega_1=\omega_2=\omega$, we write $\Aut^{\otimes,\partial}(\omega)$ (resp.  $\Aut^{\otimes}(\omega)$). In that special case, it occurs that the functor $\Aut^{\otimes,\partial}(\omega)$ (resp.  $\Aut^{\otimes}(\omega)$)  is a group functor, where the composition is given by the composition of morphisms.

\smallskip
We rephrase \cite[Proposition 4.25]{GilGorchOV} in our setting:

\begin{prop}\label{prop:diffgrouprep}
Let $\cS$ be a $\partial$-$\cC$-algebra and let $\omega :\Diff(\cF,\sgq) \rightarrow
\Proj_ \cS$ be a differential fiber functor. Let $A$ be the $\cS$-Hopf algebra
that represents the functor   $\Aut^{\otimes}(\omega)$ (see \cite[Proposition 6.19]{Delgrotfest}).
Then, $A$ has a canonical structure of $\partial$-$\cS$-Hopf algebra and represents the functor
$\Aut^{\otimes,\partial}(\omega)$.
\end{prop}

Proposition \ref{prop:diffgrouprep} shows that the functor $\Aut^{\otimes,\partial}(\omega)$ is a $\partial$-group-scheme  in the sense of Appendix \ref{appendix:diffgeometry}. If $\cS$ is a $\partial$-closed field extension of $\cC$ then one can identify
$\Aut^{\otimes,\partial}(\omega)(\cS)$ with a subgroup of $\GL_\nu(\cS)$ defined as the zero set of polynomial differential equations with coefficients in $\cS$.


\section{Galois groups}
\label{sec:Galoisgroups}

Let $(F,\sgq ,\partial)$ be a $(\sgq ,\partial)$-field.
We fix a $q$-difference module $\cM_\cF$ in $\Diff(\cF,\sgq)$
and  consider three categories \emph{generated} by $\cM_\cF$.
First of all, we consider the strictly full subcategory $\langle \cM_\cF \rangle^{\oplus}$  of $\Diff(\cF,\sgq)$,
that contains
the subquotients of all finite direct sums of copies of $\cM_\cF$, \ie, is the abelian subcategory generated by
$\cM_\cF$.
Then we need the Tannakian category $\langle \cM_\cF \rangle^{\otimes}$
(resp. differential Tannakian category $\langle \cM_\cF \rangle^{\otimes, \partial}$), that
is the strictly full Tannakian (resp. differential Tannakian) category generated by $\cM_\cF$.
It admits a very simple description:  We consider the constructions of linear (resp. linear differential) algebra
of $\cM_\cF$, i.e., the list of $q$-difference modules
$$
\bigoplus_{i,j} \cM_\cF^{\otimes i}\otimes {\cM_\cF^*}^{\otimes j}
\hskip 10 pt\l(\hbox{resp. }
\bigoplus_{i,j,l,r,s} \cM_\cF^{\otimes i}\otimes {\cM_\cF^*}^{\otimes j}
\otimes F_\partial^l(\cM_\cF^{\otimes r}\otimes{\cM_\cF^*} ^{\otimes s})\r),
$$
where $\cM_\cF^*$ denotes the dual of
$\cM_\cF$ and $i,j$ are non-negative integers
(resp. $i,j,l,r,s$ are non-negative integers and
$F_\partial^l$ the $l$-th iterate of the prolongation functor).
If  we order
the sub-objects of the constructions  of linear
(resp. linear differential) algebra of $\cM_\cF$
with respect to the relation ``be a direct summand'',
then $\langle \cM_\cF \rangle^{\otimes}$
(resp. $\langle \cM_\cF \rangle^{\otimes, \partial}$)
 is the filtering union  of the abelian categories $\langle \cN_\cF \rangle^{\oplus}$,
 where $\cN_\cF$ runs through the sub-objects of a construction of linear (resp. linear differential)
 algebra of $\cM_\cF$.
 This inductive description allows to see the Tannakian as well as
the differential Tannakian equivalence as an inductive limit of
Morita  equivalences (see \cite[Lemma 2.13]{delmil}).
Thus,  for  $\cM_\cF$ a $q$-difference module and
$\omega :\langle \cM_\cF \rangle^{\otimes, \partial} \rightarrow\Proj_{\cC}$
a differential fiber functor, we denote by
$\Aut^\otimes(\cM_\cF,\omega_{|_{\langle \cM_\cF \rangle^{\otimes}}})$
and  by $\Aut^{\otimes,\partial}(\cM_\cF,\omega)$
the groups of tensor and differential tensor automorphisms of $\omega$, respectively.

\begin{notation}\label{notation:groups}
In the current notation, the group
$\Aut^\otimes(\cM_\cF,\omega)$ would be the group of tensor of automorphisms  $\omega$
as a fiber functor defined on the category $\langle \cM_\cF \rangle^{\otimes, \partial}$,
forgetting the differential structure. Since we will never use such a group, we will make
an abuse of notation writing $\Aut^\otimes(\cM_\cF,\omega)$ for
$\Aut^\otimes(\cM_\cF,\omega_{|_{\langle \cM_\cF \rangle^{\otimes}}})$.
The same abuse of notation will be applied to other groups
defined later in the text below, unless the context requires
more precision.
\end{notation}

\subsection{The forgetful functor and the intrinsic Galois groups}

Following  \cite{andreens}, we pay particular attention to the forgetful fiber functor
$$
\eta_\cF: \Diff(\cF,\sgq) \rightarrow\Vect_\cF,
$$
that sends a $q$-difference module $\cM_\cF$ onto its underlying vector space $M_\cF$.

\begin{defn}
The intrinsic (resp. parameterized intrinsic) Galois group $\Gal(\cM_\cF)$ (resp. $\Gal^\partial(\cM_\cF)$)  is the  group
$\Aut^{\otimes}(\cM_\cF, {\eta_\cF}_{|_{\langle \cM_\cF \rangle^{\otimes}}})$
(resp. $\Aut^{\otimes,\partial}(\cM_\cF, {\eta_\cF}_{|_{\langle \cM_\cF \rangle^{\otimes,\partial}}})$).
\end{defn}

The defining equations of the  intrinsic Galois groups can be  read off from the form of the $q$-difference systems attached to $\cM_\cF$ and its construction of
linear  differential algebra. Moreover, it enjoys an arithmetic description when $K=k(q)$ with $k$ a finitely generated extension of $\Q$. This arithmetic characterization  depends  on whether $q$ is algebraic or transcendental over $\Q$. See \cite[Chapter 5 and \S7.3]{DVHARDmemAMS} for an overview of the results on this topic.
{As an example,} we present here only the result under the assumption that $q$ is transcendental over $\Q$:

\begin{thm}[{\rm\cite[Theorem 4 in the Introduction and Theorem 7.13]{DVHARDmemAMS}}]
\label{thm:curvaturecharacterization}
Let $\cM_{K(x)}$ be a $q$-difference module over $K(x)$.
The parameterized intrinsic Galois group $\Gal^\partial(\cM_{K(x)})$ is the smallest
differential algebraic subgroup
of $\GL(M_{K(x)})$, whose specialization at  $\zeta_n$ contains the specialization of
the operator $\Sgq^n$ at $\zeta_n$, for almost all positive integer $n$ and for a choice of a primitive
$n$-th root of unity $\zeta_n$.
\end{thm}
For $Y(qx)=A(x)Y(x)$, the above theorem says roughly that the set of differential algebraic equations in
$K\{Z, \det Z^{-1}\}_\partial$ defining the parameterized intrinsic Galois group
is  generated by the ones  that vanishes on the \emph{curvatures} of the system, that is on
$$
A(q^{n-1}x)\cdots A(qx)A(x)\vert_{q=\zeta_n},
$$
for almost all positive integer $n$
and for a choice $\zeta_n$ of a primitive
$n$-th root of unity.

\subsection{Fiber functors associated with weak parameterized Picard-Vessiot extensions}
\label{sec:ppvext}

In this section we show that a weak parameterized Picard-Vessiot ring naturally determines a neutral differential fiber functor.
As in the theory of Tannakian categories, we expect the contrary to be also true, but the result is not included in the literature on differential Tannakian category.
In the next section we will apply this construction to any of the rings listed in \S\ref{subsec:summaryPV}.

\begin{prop}\label{prop:wpvfiber}
Let $\cM_\cF$ be  a $q$-difference module over $\cF$ and let $R$ be a weak parameterized Picard-Vessiot ring for
a $q$-difference system $\sgq(Y)=AY$ associated to $\cM_\cF$ in some fixed basis. Then,
$$
\begin{array}{cccc} \omega_R : & \langle \cM_\cF \rangle^{\otimes,\partial} & \rightarrow & \Vect_{\cC}, \\
& \cN_\cF & \mapsto & Ker( \Sgq-id, \cN_\cF \otimes_\cF R) \end{array}
$$
is a neutral differential fiber functor.
\end{prop}

\begin{proof}
{Let $\sgq(\vec{y})=A\vec{y}$ be the $q$-difference system associated to $\cM_\cF$ in a fixed basis. }
We have $R =\cF\{Z,\det Z^{-1} \}_\partial$, where $Z \in \GL_\nu(R)$ and $\sgq(Z)=AZ$.
Hence the $q$-difference system attached to $F_\partial(\cM_\cF)$ is given by $\sgq(Y) =\begin{pmatrix} A & \partial(A) \\
0& A \end{pmatrix} Y$ and a fundamental matrix is $\begin{pmatrix} Z & \partial( Z) \\
0& Z \end{pmatrix} $.
Let $i$ be a positive integer. Repeating the argument above, we can see that
\lucia{the $q$-difference module obtained from $\cM_\cF$ iterating $i$ times the
	prolongation functor} is trivialized by $R$, \ie, admits a fundamental
solution matrix with coefficients in $R$, and that more generally
$R$ trivializes any construction $\cX_\cF$ of  differential algebra of $\cM_\cF$.
This comes from the fact that a $q$-difference system (resp. fundamental solution matrix)
attached to $\mathcal{X}_\cF$ is obtained from  $A$ (resp. $Z$) by the same construction of  differential algebra. Then, it is clear that any sub-object $\cN_\cF$ of $\mathcal{X}_\cF$ admits a
fundamental solution matrix with coefficients in $R$. Thereby, for any object
$\cN_\cF$ in $\langle \cM_\cF \rangle^{\otimes,\partial}$,
we find a functorial isomorphism between $ \cN_\cF \otimes_\cF R $ and $\omega_R(\cN_\cF)\otimes_{\cC} R$. We deduce from this fact that $\omega_R$ is a faithful, exact, $\cC$-linear tensor
functor. It is neutral because $R^{\sgq} =\cC$.  The fact that
$\omega_R$ intertwines with $F_\partial$ corresponds exactly to the fact that a fundamental solution matrix attached to $F_\partial (\cM_\cF)$ is given by the prolongation of a fundamental solution matrix attached to $\cM_\cF$, as explained above.
\end{proof}

The following proposition, which  is the parameterized analogue of \cite[Theorem 1.32.2)]{vdPutSingerDifference},
shows that  the  group $G^\partial_R$ of functorial  $\sgqp$-$\cF$-automorphism of
$R=\cF\{Z,\det Z^{-1} \}_\partial$ (see Definition \ref{defn:functorautomorphisms}) corresponds to the group of differential tensor automorphisms of the
neutral differential fiber functor $ \omega_R$, constructed in Proposition \ref{prop:wpvfiber}.

\begin{prop}\label{prop:fibergrouppvgroup}
Let $\cM_\cF$ be a $q$-difference module over $\cF$ and $R$ be a weak parameterized Picard-Vessiot ring
for a $q$-difference system attached to $\cM_\cF$. Then, the linear differential algebraic groups
$\Aut^{\otimes,\partial}(\cM_\cF,\omega_R)$ and $G_R^\partial$ are isomorphic over $\cC$.
\end{prop}

\begin{proof}
\charl{We adapt  the proof of \cite[Theorem 1.32]{vdPutSingerDifference} to our parameterized setting.
Let $\mathcal{S}$ be a $\partial$-$\mathcal{C}$-algebra.
For any $q$-difference module in $\cN_\cF$ in $\langle \cM_\cF \rangle^{\otimes,\partial}$,
{the morphism $\tau_\cS \in G_R^\partial(\cS)$ acts on $\cN_\cF \otimes_\cF R \otimes_\cC \cS$ as $\mathrm{id} \otimes \tau_\cS$.}
Since this action commutes with $\Sigma_q$, {it induces} an action $\tau_{\cN_\cF}$ of $\tau_\cS$ on $\omega(\cN_\cF)\otimes \cS$.
This defines a sequence of the form $\{ \tau_{\cN_\cF} | \cN_\cF \mbox{ object of } \langle \cM_\cF \rangle^{\otimes,\partial}\}$.
Let $f:\cN_\cF \rightarrow \mathcal{V}_\cF$ be a morphism in $\langle \cM_\cF \rangle^{\otimes,\partial}$.
Then, $f$ extends to a $R\otimes \cS$-linear map
$f\otimes \mathrm{id}: \cN_\cF\otimes_\cF R \otimes_\cC \cS  \rightarrow \mathcal{V}_\cF \otimes_\cF R \otimes_\cC \cS$,
which commutes with $\Sigma_q$ and the action of $G_R^\partial(\cS)$.
Thus, $\tau_{\cV_\cF} \circ (\omega_R(f)\otimes \mathrm{id})=  (\omega_R(f)\otimes \mathrm{id}) \circ \tau_{\cN_\cF}$.
Since $\tau_\cS$ commutes with the derivation $\partial$, we have $F_\partial(\tau_{\cN_\cF})=\tau_{F_\partial(\cN_\cF)}$.
{Moreover, $\tau_{\bold{1}}$ is clearly the identity.}
This induces a functorial group homomorphism $\alpha(\cS): G_R^\partial(\cS) \rightarrow  \Aut^{\otimes,\partial}(\cM_\cF,\omega_R)(\cS)$.}
\par
\charl{Let us prove that $\alpha(\cS)$ is injective.
If $\alpha(\cS)( \tau_\cS)$ is the identity, {then in particular,} $\tau_{\cM_\cF}$ is the identity on $\omega_R(\cM_\cF)$.
{Let $(m_i)_i$ be  a $\cF$-basis of $\cM_\cF$ and let $X:=(x_{j,i})_{ 0\leq i,j \leq \nu}$ be an matrix in $\GL_\nu(R)$,
such that $(\mu_i:=\sum_j x_{j,i} m_j)_i$ is a $\cC$-basis of $\omega_R(\cM_\cF)$.
The matrix  $X\in \GL_\nu(R)$ is a fundamental solution matrix of the system associated to $\cM_{\cF}$ in the basis $(m_i)_i$,
whose  coefficients generate $R$ as $\partial$-$\cF$-algebra. Notice that $\tau_\cS(X)=X$, since $\tau_{\cM_\cF}$ acts as the identity on $\omega_R(\cM_\cF)$.
Therefore,
$\tau_\cS$ is the identity on $R\otimes \cS$, which proves that $\alpha(\cS)$ is injective.}}
\par
\charl{Conversely, let $\tau=\{\tau_{\cN_\cF} |\cN_\cF \mbox{ object of } \langle \cM_\cF \rangle^{\otimes,\partial} \}$ be an element of $\Aut^{\otimes,\partial}(\cM_\cF,\omega_R)(\cS)$. We want to construct an element $\tau_\cS \in G_R^\partial(\cS)$ such that $\alpha(\cS)(\tau_\cS)= \tau$.
Let us write $R$ as {$\cF\{X, \frac{1}{\det X}\}_\partial$}.
The action of $\tau_{\cM_\cF}$ in the $\cS$-basis $\mu_1\otimes 1, \dots,\mu_\nu \otimes 1$ is given
by an invertible matrix $[\tau]_\cS \in \GL_\nu(\cS)$. We  consider the morphism $\tau_\cS$ of $\cF$-algebra
of $R \otimes \cS$ defined as follows:  $\tau_\cS(X)=X.[\tau]_\cS$,
{$\tau_\cS( X^{(h)}) =\partial^h(X[\tau]_\cS)$ for any non-negative integer $h$}.
The  morphism $\tau_\cS$ is well defined  if  for any  differential polynomial $P$ such that  $P(X)=0$ we have $\tau_\cS(P(X))=0$.
A differential algebraic relation for the fundamental solution matrix $X$ can be seen as a $\cF$-linear form
that annihilates on a construction  $\cN_\cF$ of linear differential algebra of $\cM_\cF$.
Since the set of $\cF$-linear forms that vanish  on $\cN_\cF$ is a $q$-difference submodule
of $\cN_\cF^*$, it must be stabilized by $\tau$.  It follows
that $\tau_\cS(P(X))=0$ for any differential polynomial $P$ such that $P(X)=0$.
{One can check that} the compatibility of the sequence $\tau$ with the tensor product and the prolongation
functor implies  that $\alpha(\cS)(\tau_\cS)=\tau$.}
\par
\charl{To conclude, we have proved that for any $\partial$-$\cC$-algebra, the $\alpha(\cS)$'s are isomorphisms.
This proves that $\Aut^{\otimes,\partial}(\cM_\cF,\omega_R)$ and $G_R^\partial$ are isomorphic over $\cC$.  }
\end{proof}



As in Remark \ref{rmk:AssumptionsOnTheBaseField}, we consider a finitely generated extension $K/\Q$, an element $q\in K\smallsetminus\{0\}$ and a field embedding
$K\hookrightarrow\C$, such that $|q|\neq 1$ for the usual norm of $\C$.
Let $\cM_{K(x)}$ be a $q$-difference module over $K(x)$.

\begin{notation}\label{notation:scalarextension}
For any $q$-difference field extension $\cF/K(x)$
we will denote by $\cM_\cF$ the $q$-difference module over $\cF$ obtained
from $\cM_{K(x)}$ by scalar extension. More precisely, $M_\cF=M_{K(x)}\otimes_{K(x)} \cF$ and $\Sgq$ is defined on $\cM_\cF$ by $ \Sgq \otimes \sgq$.
\end{notation}

Let $\sgq(\vec{y})=A\vec y$ be the $q$-difference system associated to $\cM_{K(x)}$ with respect to a fixed basis.
We will consider as in \S\ref{subsec:summaryPV} the weak Picard-Vessiot rings
$R$, $R_E$ and $\wtilde R$, extending conveniently the constants to $\C$, $C_E$ and $\wtilde C_E$ respectively.
In parallel, following Proposition \ref{prop:wpvfiber}, each of these weak parameterized Picard-Vessiot rings
yields to a neutral  differential fiber functor for
$\langle\cM_{\wtilde \C_E(x)} \rangle^{\otimes, \partial}$,
$\langle\cM_{\C(x)} \rangle^{\otimes,\partial}$,
and $\langle\cM_{\C_E(x)} \rangle^{\otimes,\partial}$, respectively.
When restricted to the Tannakian category
$\langle\cM_{\wtilde \C_E(x)} \rangle^\otimes$,
$\langle\cM_{\C(x)} \rangle^\otimes$,
and $\langle\cM_{\C_E(x)} \rangle^\otimes$,
these differential fiber functors induce neutral fiber functors in the classical sense of \cite{Delgrotfest}.
Proposition \ref{prop:fibergrouppvgroup} immediately implies the following:

\begin{cor}\label{cor:PVgroupsVStannakiangroups}
We have:
\begin{itemize}
  \item $G^\partial_R\cong\Aut^{\otimes,\partial}(\cM_{\C(x)},\omega_R)$ over $\C$.
  \item $G^\partial_{R_E}\cong\Aut^{\otimes,\partial}(\cM_{\C_E(x)},\omega_{R_E})$ over $\C_E$.
  \item $G^\partial_{\wtilde R}\cong\Aut^{\otimes,\partial}(\cM_{\wtilde\C_E(x)},\omega_{\wtilde R})$ over $\wtilde\C_E$.
\end{itemize}
\end{cor}

\subsection{List of all fiber functors}
\label{subsec:listfiberfunctors}

For the reader convenience we remind the list of all neutral differential fiber functors defined above:

\begin{center}
\begin{tabular}{ll}
  $\om_R:\langle\cM_{\C(x)}\rangle^{\otimes,\partial}\longrightarrow \Vect_{\C}$,&
    $\cN\mapsto \ker(\Sgq-Id, R\otimes_{ \C(x)} \mathcal{N})$,\\&\\
  $\om_{R_E}:\langle\cM_{\C_E(x)}\rangle^{\otimes,\partial}\longrightarrow \Vect_{\C_E}$,&
    $\cN\mapsto \ker(\Sgq-Id, R_E\otimes_{ \C_E(x)} \mathcal{N})$,\\&\\
  $\om_{\wtilde R}:\langle\cM_{ \wtilde \C_E (x)}\rangle^{\otimes,\partial}\longrightarrow \Vect_{\wtilde \C_E}$,&
    $\cN\mapsto \ker(\Sgq-Id, \wtilde R\otimes_{ \wtilde \C_E(x)} \mathcal{N})$,\\
\end{tabular}
\end{center}
whose associated groups are $G^\partial_R$, $G^\partial_{R_E}$ and
$G^\partial_{\wtilde R}$.
Moreover we have the four forgetful functors:
\begin{center}
\begin{tabular}{l}
$\eta_{K(x)}:\langle\cM_{K(x)}\rangle^{\otimes,\partial}\longrightarrow \Vect_{K(x)}$,\\\\
$\eta_{\C(x)}:\langle\cM_{\C(x)}\rangle^{\otimes,\partial}\longrightarrow \Vect_{\C(x)}$,\\\\
$\eta_{\C_E(x)}:\langle\cM_{\C_E(x)}\rangle^{\otimes,\partial}\longrightarrow \Vect_{\C_E(x)}$,\\\\
$\eta_{\wtilde \C_E (x)}:\langle\cM_{\wtilde \C_E (x)}\rangle^{\otimes,\partial}\longrightarrow \Vect_{\wtilde \C_E}$,\\
\end{tabular}
\end{center}
that define the intrinsic Galois groups
$\Gal^\partial(\cM_{K(x)})$, $\Gal^\partial(\cM_{\C(x)})$,
$\Gal^\partial(\cM_{\C_E(x)})$ and $\Gal^\partial(\cM_{\wtilde C_E(x)})$, respectively.
We will call by the same name the restrictions of the \lucia{functors} above to the usual Tannakian categories
$\langle\cM_{K(x)} \rangle^\otimes$,
$\langle\cM_{\C(x)} \rangle^\otimes$,
$\langle\cM_{\C_E(x)} \rangle^\otimes$,
and $\langle\cM_{\wtilde \C_E(x)} \rangle^\otimes$. Using  Notation \ref{notation:groups} for the groups, i.e., dropping the superscript $\partial$,
we obtain the following Tannakian groups:
$\Gal(\cM_{K(x)})$, $\Gal(\cM_{\C(x)})$,
$\Gal(\cM_{\C_E(x)})$, $\Gal(\cM_{\wtilde C_E(x)})$, respectively.
One can consider the difference Galois groups $G_R$, $G_{R_E}$, $G_{\wtilde R}$, {defined in Remark \ref{rmk:classicalPV}}.
Notice that the analogue of Corollary \ref{cor:PVgroupsVStannakiangroups} for $G_R$, $G_{R_E}$, $G_{\wtilde R}$
is well known (see \cite[\S9.4]{Delgrotfest}).

\section{Comparison theorems}

One of the main results of \cite[\S 3]{ChatHardouinSinger}  is
(see also \cite{peonnietodiffcomp}, for a model theoretic approach):

\begin{thm}\label{thm:compalgpv}
The group schemes $G_R$,  $G_{R_E}$ and
$G_{\wtilde R}$
become isomorphic over $\wtilde \C_{E}$.
\end{thm}

\begin{rmk}
In \cite{Sauloy-Galois-AENS}, Sauloy constructs a $\C$-linear fiber functor for $q$-difference modules over $\C(x)$, using a basis of meromorphic solutions.
Since $\C$ is algebraically closed, it follows from the classical general theory of Tannakian categories, that such a fiber
functor gives rise to a group that is isomorphic to the Picard-Vessiot group of \cite{vdPutSingerDifference} over $\cF=\C(x)$.
We won't consider Sauloy's point of view in this paper.
\end{rmk}

One of the most important properties of functional Galois groups is that their dimension as algebraic variety is equal to the transcendence degree of the associated Picard-Vessiot rings. In particular, the {sets of} the entries of any fundamental solution matrix of \eqref{eq:sys} in $R$, $R_E$ or $\wtilde R$ have the same transcendence degree over the associated base field, i.e., $\C(x)$, $C_E(x)$ or $\wtilde\C_E(x)$, respectively.
\par
We have the following differential analogue of the Theorem above:

\begin{thm}\label{thm:weakextension}
In the previous notation, $G^\partial_R\otimes_\C \wtilde \C_E \simeq
G^\partial_{R_E}\otimes_{\C_E} \wtilde\C_E \simeq
G^\partial_{\wtilde R}$.
\end{thm}

\begin{rmk}
\lucia{The proof below is a parameterized analog of  \cite[Corollary 2.5]{ChatHardouinSinger}.}
\end{rmk}

\begin{proof}
\charl{Let
$S:= \C_E(x)\{Y, \det Y^{-1} \}_\partial/
\mathfrak{q}\C_E(x)$ be the PPV ring over $\C_E(x)$ defined  as in Corollary \ref{cor:isoppv} and let $\phi:R\otimes_\C \C_E \rightarrow S$ be the embedding considered in the proof of Proposition \ref{prop:pvextell}.  The group $G^\partial_R$ is a functor from $\partial$-$\C$-algebras $A$ to groups defined
by $G^\partial_R(A)=\Aut_{\C(x)\otimes A}^{\sgqp}(R \otimes_\C A)$.  We define analogously $G^\partial_S$ as a functor from $\partial$-$\C_E$-algebras to groups. By Proposition \ref{prop:diffgrouprep}, these functors are representable. {(See Appendix \ref{appendix:diffgeometry}.)}
Let $T_R$ be the finitely generated $\partial$-$\C$-algebra representing
$G^\partial_R$ and let  $T_S$ be the finitely generated $\partial$-$\C_E$-algebra representing
$G^\partial_S$.  We define a new functor $F$ from $\partial$-$\C_E$-algebras $B$ to groups as $F(B) =\Aut_{\C_E(x) \otimes B}^{\sgqp}((R \otimes_\C \C_E)\otimes_{\C_E} B)$. One can easily check that $F$ is representable by $T_R\otimes_\C \C_E$. Using the embedding $\phi$, one sees that
\begin{equation}
F(B)=\Aut_{\C_E(x) \otimes B}^{\sgqp}(S\otimes_{\C_E} B)= G_R(B),
\end{equation}
for any $\partial$-$\C$-algebra $B$. Yoneda Lemma {(see Appendix \ref{appendix:diffgeometry})}
yields to $T_R \otimes_\C \C_E  \simeq T_S$, which is $G^\partial_R \otimes_\C \C_E \simeq G^\partial_S$.
A similar {argument} shows that the isomorphism  of $   \sgqp$-$\widetilde \C_E$-algebras between $S\otimes_{\C_E} \wtilde \C_E$ and $R_E\otimes_{\C_E} \wtilde \C_E$  yields to the isomorphism   $G^\partial_S\otimes_\C \wtilde \C_E \simeq
G^\partial_{R_E}\otimes_{\C_E} \wtilde\C_E$. This proves that $G^\partial_R\otimes_\C \wtilde \C_E \simeq
G^\partial_{R_E}\otimes_{\C_E} \wtilde\C_E$.}
\par
\lucia{Replacing $S$ with $\widetilde S$ (see Corollary \ref{cor:isoppv}), one shows in the same way
that $G^\partial_R\otimes_\C \wtilde \C_E \simeq
G^\partial_{ \widetilde R}\otimes_{\C_E} \wtilde\C_E$.}
\end{proof}

\begin{rmk}
By \cite[Prop. 6.21]{HardouinSinger},  the  Zariski closure of $G^\partial_R$, $G^\partial_{R_E}$ and
$G^\partial_{\wtilde R}$ coincide with $G_R$, $G_{R_E}$ and
$G_{\wtilde R}$, {respectively. Therefore} we can retrieve the Theorem \ref{thm:compalgpv} as a corollary of Theorem \ref{thm:weakextension}.
\end{rmk}


We are now concerned with the intrinsic Galois groups,
{both parameterized and not}.
Let $\cM_{K(x)}$ be a $q$-difference module defined over $K(x)$,
with $K$ a finite generated extension of $\Q$.
For a $q$-difference module $\cM_\cF$ over $\cF$, the  comparison between the  intrinsic Galois group  and  the group of tensor automorphism of a  neutral fiber functor $\omega$ for $\langle \cM_\cF \rangle^\otimes$ is a direct consequence of the fact that   $\Hom^{\otimes}(\omega, \eta_\cF)$, which is a bitorsor on $\Aut^{\otimes}(\cM_\cF,\omega)$ and $\Gal(\cM_\cF)$, is also an $\cF$-scheme and has therefore a point in some algebraically closed extension $\widetilde{\cF}$ of $\cF$. This point gives rise to an isomorphism over $\widetilde{\cF}$ between $\Aut^{\otimes}(\cM_\cF,\omega)$ and $\Gal(\cM_\cF)$.   A similar result holds in the differential parameter context.  More, precisely, \cite[Proposition 4.25]{GilGorchOV}) shows that, when $\omega$ is a neutral differential fiber functor for  $\langle \cM_\cF \rangle^{\otimes,\partial}$, the functor  $\Hom^{\otimes,\partial}(\omega, \eta_\cF)$ is a $\partial$-$\cF$-scheme. As above, this yields an isomorphism between $\Aut^{\otimes,\partial}(\cM_\cF,\omega)$ and $\Gal^\partial(\cM_\cF)$ over a differentially closed field extension of $\cF$. In our $q$-difference setting,
this leads to the follwing statement:

\begin{prop} \label{prop:compfibreoubli}
 Let us denote by $\wtilde{\C(x)}$ (resp. $\wtilde{\C_E(x)}$) a differential closure of $\C(x)$ (resp. $\C_E(x)$).
 We have the following isomorphisms of group-schemes:
\begin{enumerate*}
\item
$\Aut^{\otimes}(\cM_{\C(x)}, \omega_R) \otimes_{\C}\wtilde{\C(x)}  \simeq
\Gal(\cM_{\C(x)}) \otimes_{\C(x)} \wtilde{\C(x)}$;
\item $ \Aut^{\otimes}(\cM_{\C_E(x)}, \om_{R_E}) \otimes_{\C_E} \wtilde{\C_E(x)} \simeq \Gal(\cM_{\C_E(x)}) \otimes_{\C_E(x)}\wtilde{\C_E(x)}$;
\end{enumerate*}
and the following isomorphisms of $\partial$-group-schemes:
\begin{enumerate*}
\item[1bis.]
$\Aut^{\otimes,\partial}(\cM_{\C(x)}, \omega_R) \otimes_{\C}\wtilde{\C(x)}  \simeq
\Gal^\partial(\cM_{\C(x)}) \otimes_{\C(x)} \wtilde{\C(x)}$;
\item[2bis.]
$ \Aut^{\otimes,\partial}(\cM_{\C_E(x)}, \om_{R_E}) \otimes_{\C_E} \wtilde{\C_E(x)} \simeq \Gal^{\partial}(\cM_{\C_E(x)}) \otimes_{\C_E(x)} \wtilde{\C_E(x)}$.
\end{enumerate*}
\end{prop}

Since the dimension of a $\partial$-group-scheme as well as the differential
transcendence degree (see Definition \ref{defn:differential transcendence degree})
of a field extension is stable up to field extension, one obtains the following corollary:

\begin{cor}\label{cor:caldifftrans}
Let $\cM_{K(x)}$ be a $q$-difference module defined over $K(x)$. Let $U \in \GL(\cM  (\C^\ast))$
be a fundamental solution matrix {attached to $\cM_{K(x)}$}, as in Proposition \ref{thm:praagman}.

Then, the differential transcendence degree of the differential  field $F_E$ generated  over  $ \C_E (x)$
by the entries of $U$
is equal to the differential dimension of $\Gal^{\partial}(\cM_{ \C (x)})$ over $\C(x)$.
\end{cor}

\begin{proof}
By \cite[Proposition 4.28]{GilGorchOV}, the functor $\Isom^{\otimes,\partial}(\omega_{R_E}\otimes \C_E(x),\eta_{\C_E(x)})$ is a reduced $\partial$-$\C_E(x)$-scheme, represented by $R_E$. It is also a $\Aut^{\otimes,\partial}(\cM_{\C_E(x)}, \om_{R_E})$-torsor. It has thus a $\wtilde{\C_E(x)}$-point, which
gives, by triviality of the torsor,   a   $\sgqp$-isomorphism between
$\wtilde{ \C_E(x)} \otimes_{\C_E(x)} R_E $ and $\wtilde {\C_E(x)} \otimes_{\C_E} \C_E\{\Aut^{\otimes,\partial}(\cM_{\C_E(x)}, \om_{R_E}) \}$. Using  the discussion on the differential dimension in Appendix \ref{appendix:diffgeometry}, we get that the differential dimension of
$\Aut^{\otimes,\partial}(\cM_{\C_E(x)}, \om_{R_E})$ equals the differential transcendence degree of $F_E$ over $\C_E(x)$. By  Proposition \ref{prop:compfibreoubli} combined with Proposition \ref{thm:weakextension}, we find
that $\Aut^{\otimes,\partial}(\cM_{\C_E(x)}, \om_{R_E})$ is isomorphic to $\Gal^{\partial}(\cM_{\C
(x)})$ over $\wtilde \C_E(x)$. We conclude by using one more time the fact that the
differential dimension of a reduced $\partial$-scheme is invariant by base field extension.
\end{proof}

In \cite[Lemma 1.3.2]{katzcal}, it is shown that the group of tensor
automorphisms of a $K$-linear neutral fiber functor is invariant up to algebraic field
extension of $K$. For forgetful functors, this is not true.
This is essentially due to the fact that, unlike to the case of neutral fiber functors,
a vector space stable under the action of the group of tensor automorphisms of the forgetful functor is  not necessarily an
object of the Tannakian category.
{However, one can show that, for any field extension $L/K$, the
parameterized intrinsic Galois group of $\cM_{L(x)}$ is equal, up to
scalar extension}, to the parameterized intrinsic Galois group of
$\cM_{K^\p(x)}$, for a {suitable} finitely generated extension
$K^\p/K$, with $K^\p\subset L$.\footnote{In \cite{barkatou-cluzeau-divizio-weil-reduced},
for differential modules, the authors optimize the field
on which such an isomorphism is true, using an effective characterization of Kolchin's reduced forms.}

\begin{lemma}\label{lemma:reduccourb}
Let $L$ be a field extension of $K$ with $\sgq{|_{L}}= id$. There
exists a finitely generated intermediate field $L/K^\p/K$ such that
$$
\Gal(\cM_{L(x)}) \cong \Gal(\cM_{K'(x)}) \otimes_{K'(x)} L(x)\, $$
 and
 $$
\Gal^{\partial}(\cM_{L(x)}) \cong \Gal^{\partial}(\cM_{K'(x)}) \otimes_{K'} L(x).
$$
These equalities hold when we replace $K^\p$ by
any subfield extension of $L$ containing $K^\p$.
\end{lemma}

\begin{proof}
By definition, $\Gal^{\partial}(\cM_{L(x)})$
is the stabilizer inside
$\GL(M_{L(x)})$ of all $L(x)$-vector spaces of the form $W_{L(x)}$ for
$\cW$ object of
$\langle\cM_{L(x)}\rangle^{\otimes,\partial}$.

Similarly, for any field extension $L/K^\p/K$, we have an equality
$$
\Gal^{\partial}(\cM_{K^\p(x)}) =Stab(W_{K^\p(x)},\cW \; \mbox{object of } \; \langle\cM_{K^\p(x)} \rangle^{\otimes,\partial}),
$$
that has to be understood as
a functorial equality for differential scheme defined above $L(x)$.
Then,
$$
\Gal^{\partial}(\cM_{L(x)}) \subset \Gal^{\partial}(\cM_{K^\p(x)}) \otimes L(x).$$
By noetherianity, the (parameterized) intrinsic Galois group of $ \cM_{L(x)}$ is defined
by a finite family of (differential) polynomial equations, thus we can choose
$K^\p$, which contains the coefficients of the defining equations.
\end{proof}

The corollary below summarizes results of this chapter.

\begin{cor}\label{cor:diffalgrelation}
Let $\cM_{K(x)}$ be a $q$-difference module defined over $K(x)$. Let $ U \in \GL_{\nu} (\cM(C^\ast))$ be a fundamental matrix of meromorphic solutions of $\cM_{K(x)}$.
Then,
\begin{enumerate*}
\item
the dimension of $\Gal(\cM_{\C(x)})$ is equal to the transcendence degree of the field generated by the
entries of $U$ over  $\C_E (x)$, \ie, the $\C(x)$-group-scheme
$\Gal(\cM_{\C(x)})$ measures the algebraic relations between the meromorphic solutions of $\cM_{\C_E(x)}$.
\item
the differential dimension of  $\Gal^{\partial}(\cM_{C(x)})$ is equal to the
differential transcendence degree of the differential field generated by the entries of $U$ over  $\wtilde \C_E (x)$, \ie, the
$\partial$-$\C(x)$-group-scheme  $\Gal^{\partial}(\cM_{\C(x)})$ encodes the differential algebraic relations
between the meromorphic solutions of $\cM_{K(x)}$.

\item
there exists a finitely generated extension $K^\p/K$ such that the
differential transcendence degree of
the differential field generated by the entries of $U$ over  $\wtilde{\C}_E (x)$ is equal to the differential dimension of
$\Gal^{\partial}(\cM_{K^\p(x)})$, \ie, it is given by an arithmetic characterization (see Theorem \ref{thm:curvaturecharacterization}).
\end{enumerate*}
\end{cor}

\begin{proof}
The first two statements are proved in Corollary \ref{cor:caldifftrans}.
The third one is Lemma \ref{lemma:reduccourb}.
\end{proof}


\Appendix

We quickly recall some basic facts of differential
algebra as well as some very basic notions of differential algebraic geometry,
mainly in the affine case. We refer to \cite{kolchin-diff-algebra} and \cite{Kovdiffscheme} for a detailed exposition.

\section{Differential algebra}\label{app:diffalgebra}

We largely use standard notation of  differential algebra as can
be found in  \cite{kolchin-diff-algebra}.
A differential ring (or $\partial$-ring for short) is a ring $R$ together with a derivation $\partial:R\rightarrow R$,
\ie, {an additive} map  $\partial : R \rightarrow R$ satisfying the Leibniz rule
$\partial(ab)=\partial(a)b + a \partial(b)$, for all $(a,b) \in R^2$. The ring of $\partial$-constants of $R$ is $R^\partial=\{r\in R| \ \partial(r)=0\}$.
All rings considered in this work are commutative with identity and
all \emph{differential}  rings contain the ring of integer numbers.
In particular, all fields are of characteristic zero.
\par
Given two $\partial$-rings $(R,\partial)$ and ($R^\p,\partial^\p)$,
a morphism $\psi:R\rightarrow R'$ of $\partial$-rings is a morphism of rings such that $\psi\partial=\partial'\psi$.

A $\partial$-ideal $\mathfrak I$ of a $\partial$-ring
$R$ is an ideal of $R$ that is invariant under the action of $\partial$. A $\partial$-ring $R$ is said to be $\partial$-simple if it does not contain any non-zero proper $\partial$-ideals.

A $\partial$-field $k$ is a field that is also a $\partial$-ring. A $\partial$-$k$-algebra
$R$ is a $k$-algebra and a $\partial$-ring such that the morphism $k\to R$ is a morphism of $\partial$-rings. Given two $\partial$-$k$-algebras $(R,\partial)$ and ($R^\p,\partial^\p)$,
a morphism $\psi:R\rightarrow R'$ of $\partial$-$k$-algebras is a morphism of $k$-algebras such that $\psi\partial=\partial'\psi$.
If, moreover,  $R$ is a $\partial$-field and a $\partial$-$k$-algebra, we say that $R|k$ is a  $\partial$-field extension.

Let $k$ be a $\partial$-field and $R$ a $\partial$-$k$-algebra.
If $B$ is a subset of $R$,
then $k\{B\}_\partial$ denotes the smallest $\partial$-$k$-subalgebra of $R$ that contains $B$.
If $R=k\{B\}_\partial$ for some finite subset $B$ of $R$, we say that $R$ is finitely $\partial$-generated over $k$.
If $K|k$ is an extension of $\partial$-fields and $B \subset K$, then $k\left<B\right>_\partial$ denotes the smallest $\partial$-field
extension of $k$ inside $K$ that contains $B$.

\begin{defn}\label{defn:differential-polynomials}
The $\partial$-$k$-algebra $k\{x\}_\partial=k\{x_1,\ldots,x_n\}_\partial$ of $\partial$-polynomials over $k$
in the $\partial$-variables $x_1,\ldots,x_n$ is the polynomial ring over $k$ in the
countable set of algebraically independent variables $x_1,\ldots,x_n,\partial(x_1),\ldots,\partial(x_n),\ldots,$ with an action of $\partial$
as suggested by the names of the variables.
\end{defn}

Of course, for any $\partial$-field
extension $L|k$ and any  $f:=(f_1,\dots,f_n) \in L^n$, one has a $\partial$-$k$-morphism from
$k\{x\}_\partial$ to $L$, which assigns $x_i$ to $f_i$, for all $i=1,\dots,n$.
We say that $f$ is a solution of the differential algebraic equation $P(x)=0$, for some
$P \in k\{x\}_\partial$, if $P$ lies in the kernel of the specialization morphism above.

\begin{defn}\label{defn:differentially-closed-field}
A $\partial$-field $k$ is called differentially closed or $\partial$-closed, for short,
if  any system of differential algebraic equations with coefficients in $k$, having a solution in some
differential field extension of $k$, has a solution in $k$.
A differential closure of a $\partial$-field $k$
is a $\partial$-field extension of $k$ that is  $\partial$-closed and  that embeds, as $\partial$-field extension of $k$, in any differentially closed extension of $k$.
\end{defn}

\begin{defn}\label{defn:differential transcendence degree}
Let $L|K$ be a $\partial$-field extension. Elements $a_1,\ldots,a_n\in L$ are called {differentially
(or $\partial$-algebraically) independent over $K$} if the elements $a_1,\ldots,a_n,\partial(a_1),\ldots,\partial(a_n),\ldots$
are algebraically independent over $K$. Otherwise, they are called differentially dependent over $K$.
\par
{A {$\partial$-transcendence basis of $L$ over $K$} is a subset of $L$, which is maximal with respect to the property of being  a differentially independent set over $K$.}
\par
Any two $\partial$-transcendence basis of $L|K$ have the same cardinality and so we can define the
{$\partial$-transcendence degree of $L|K$}
(or differential transcendence degree of $L|K$, when the choice of $\partial$ is clear,
or also $\partial$-$ \operatorname{trdeg}(L|K)$, for short)
as the cardinality of any $\partial$-transcendence basis of $L$ over $K$.
\end{defn}

\section{Differential geometry}\label{appendix:diffgeometry}
\label{sec:diffalgeo}

In this paper, we work with the formalism of affine  differential group schemes, as can be found
 in \cite{Kovdiffscheme}. In this section, we fix  a $\partial$-field $k$ of characteristic zero,
 not necessarily $\partial$-closed.
 We define a $\partial$-$k$-scheme as follows:

\begin{defn}
An affine  \lucia{$\partial$-scheme over $k$ (or an affine $\partial$-$k$-scheme, for short)} is a (covariant) functor from the category
of $\partial$-$k$-algebras to the category of sets which is representable.
\end{defn}

The definition above means that a functor $X$ from the category of $\partial$-$k$-algebras to the category of sets is a $\partial$-$k$-scheme
if and only if there exists a $\partial$-$k$-algebra $k\{X\}$ and an isomorphism of functors
$ X\simeq \Alg_k^\partial(k\{X\},-), $
where $\Alg_k^\partial$ stands for morphism of $\partial$-$k$-algebras. By the Yoneda lemma, the $\partial$-$k$-algebra $k\{X\}$ is uniquely determined up to unique $\partial$-$k$-isomorphisms. We call it the { ring of $\partial$-coordinates of $X$}. A $\partial$-$k$-scheme $X$ is called {$\partial$-algebraic} (over $k$) if $k\{X\}$ is finitely $\partial$-generated over $k$.
We say that a $\partial$-$k$-scheme $X$ is {reduced} if $k\{X\}$ has no non-zero nilpotent elements.

Let $X$ be a $\partial$-$k$-scheme. By a {closed $\partial$-$k$-subscheme $Y\subset X$}
we mean a subfunctor $Y$ of $X$ which is represented by $k\{X\}/\I(Y)$ for some $\partial$-ideal $\I(Y)$ of $k\{X\}$.
The ideal $\I(Y)$ of $k\{X\}$ is uniquely determined by $Y$ and vice versa. We call it the {vanishing ideal of $Y$} in $X$.

A morphism of $\partial$-$k$-schemes is a morphism of functors. If $\phi\colon X\to Y$ is a morphism
of $\partial$-$k$-schemes, we denote the dual morphism of $\partial$-$k$-algebras with $\phi^*\colon k\{Y\}\to k\{X\}$.

If a functor (resp. $\partial$-functor) $X$ factors through the category of group, we say that $X$ is
a {$k$-group-scheme (resp.  $\partial$-$k$-group-scheme)}. We denote by $\mathrm{GL}_\nu(k)$ the $\partial$-$k$-groupscheme attached to  the general linear group of size $\nu$ over $k$.  It is represented by
the $\partial$-$k$-algebra $k\l\{X,\det X^{-1}\r\}_\partial$ where $X$ is a $\nu \times \nu$ matrix of $\partial$-indeterminates. More generally, for any $k$-vector space $V$ of finite dimension, we denote by $\mathrm{GL}(V)$ the $\partial$-$k$-groupscheme of invertible $k$-linear automorphisms of $V$.
{Notice that we are calling $\mathrm{GL}_\nu(k)$ both the $k$-groupscheme and the $\partial$-$k$-groupscheme attached to  the general linear group,
anyway the context will always make clear to which one of the two structures we are referring to, without introducing complicate notation.}
\par
By a {$\partial$-subgroup $H$ of $G$}, we mean a  $\partial$-$k$-scheme $H$ \charl{such that $H(S)$ is a subgroup of $ G(S)$ for  every $\partial$-$k$-algebra $S$}. We call $H$ {normal} if $H(S)$ is a
normal subgroup of $G(S)$ for every $\partial$-$k$-algebra $S$. As in the classical setting, Yoneda lemma implies that, for a $\partial$-$k$-group-scheme $G$, the algebra $k\{G\}$ is
a $\partial$-$k$-Hopf algebra, i.e.,  a $\partial$-$k$-algebra equipped with the structure of a Hopf algebra
over $k$ such that the Hopf algebra structure maps are morphisms of $\partial$-rings. It also follows immediately that  the category of $\partial$-$k$-group-schemes is anti-equivalent
to the category of $\partial$-$k$-Hopf algebras. Then, since
 Hopf algebras over fields of characteristic zero are reduced  by \cite[Cartier's Theorem in \S 11.4]{Watschem}, we get that any $\partial$-$k$-group-scheme is automatically reduced.
Reduced  $\partial$-schemes correspond to differential varieties in the sense of Kolchin  (see for instance
\cite{kolchin-diff-algebra}), for whom it suffices to focus on the solution set of
a system of differential equations with value in a sufficiently big field, i.e., a $\partial$-closed field.

The $\partial$-schemes considered in this paper are all reduced.
Thus, we  only  define the {differential dimension} of a reduced  $\partial$-scheme.
So let $V$ be a reduced $\partial$-$k$-scheme.
We can write $k\{V\}=k\{x_1,\dots,x_n\}_\partial/\mathfrak{q}$
for some positive integer $n$ and some radical $\partial$-ideal $\mathfrak{q} \subset k\{x_1,\dots,x_n\}_\partial$.
Since $\mathfrak{q}$ is radical, by \cite[Theorem 7.5]{Kapldiffalg} there exists finitely many  prime
$\partial$-ideals $\mathfrak{p}_i$ such that $\mathfrak{q}= \cap \mathfrak{p}_i$.
Now, we can define the differential dimension of $V$ over $k$, denoted by $\partial$-$\operatorname{dim}(V|k)$ as the maximum of the
$\partial$-$\operatorname{trdeg}(L_i|k)$ where $L_i$ denotes the  fraction field of $k\{x_1,\dots,x_n\}_\partial/\mathfrak{p}_i$.
In \cite[III.\S 6.Proposition 3]{kolchin-diff-algebra}, Kolchin proved
that if $k\subset k'$ is an extension of $\partial$-field and if
V is a reduced $\partial$-$k$-scheme, then $\partial$-$\operatorname{dim}(V|k)=\partial$-$\operatorname{dim}(V_{k'}|k')$,
where $V_{k'}$ is the base extension of $V$ to $k'$.

Let $V$ be a $k$-scheme, \ie, a (covariant) functor from the category of
$k$-algebras to the category of sets which is representable by a $k$-algebra $k[V]$.
We call $k[V]$ the ring
of coordinates of $V$. In \cite{Gilletdiffalg}, the author shows that
the forgetful functor
$$
\eta : \hbox{~$\partial$-$k$-algebras~} \rightarrow \hbox{~$k$-algebras},
$$
that associates to any $\partial$-$k$-algebra its underlying $k$-algebra,  has a left
adjoint denoted by $D$. This implies that the functor $\mathbf V $  from the category of
$\partial$-$k$-algebras  to the category of Sets, defined by the composition of $V$ with the forgetful functor $\eta$
is a  $\partial$-$k$-scheme, whose ring of $\partial$-coordinates is precisely $D(k[V])$. We call
$\mathbf V$, the $\partial$-$k$-scheme attached to $V$. The simple idea behind this construction
is that polynomial equations are $\partial$-polynomials. More precisely if  $V \subset \mathbf{A}_k^n$, the affine space of dimension $n$ over $k$, and
if $I(V) \subset k[x_1,\dots,x_n]$ is the vanishing ideal of $V$ as subscheme of  $\mathbf{A}_k^n$ then
the vanishing ideal of $\mathbf V$ as $\partial$-$k$-subscheme of $\mathbf{A}_k^n$ is nothing else than
the $\partial$-ideal generated by $I(V)$ in $k\{x_1,\dots,x_n\}_\partial$. Finally, Kolchin
irreducibility theorem states that if $k[V]$ is a finitely generated integral $k$-algebra,
then  $D(k[V])$ is a finitely $\partial$-generated integral $\partial$-$k$-algebra and
 the dimension of $V$ as $k$-scheme coincides with the $\partial$-dimension of $\mathbf V$ over $k$ (\cite[\S 2]{Gilletdiffalg}).

Conversely, given a $\partial$-$k$-subscheme $\mathbf V$ of some $\mathbf{A}_k^n$, we can
attach to $\mathbf V$ a $k$-subscheme of $\mathbf{A}_k^n$ as follows. Let
$\I(\mathbf V) \subset k\{x_1,\dots,x_n \}_\partial$ be the vanishing ideal
of $\mathbf V$ in $\mathbf{A}_k^n$. Let   ${\mathbf V}^Z$ be the $ k$-subscheme  of $\mathbf{A}_k^n$
defined by the ideal $\I(V) \cap k[x_1,\dots,x_n]$. We call  ${\mathbf V}^Z$
the Zariski closure of $\mathbf V$ inside $\mathbf{A}_k^n$.  If $k$ is $\partial$-closed then ${\mathbf V}^Z$ is the closure of
$\mathbf V$ with respect to the Zariski topology.



\newcommand{\noopsort}[1]{}
\providecommand{\bysame}{\leavevmode\hbox to3em{\hrulefill}\thinspace}
\providecommand{\MR}{\relax\ifhmode\unskip\space\fi MR }
\providecommand{\MRhref}[2]{%
  \href{http://www.ams.org/mathscinet-getitem?mr=#1}{#2}
}
\providecommand{\href}[2]{#2}

\end{document}